\documentclass[reqno]{amsart}
\usepackage{hyperref}
\usepackage{amssymb}
\usepackage[english]{babel}
\usepackage[T1]{fontenc}
\usepackage[ansinew]{inputenc}
\usepackage{color}

\usepackage{graphicx}
\usepackage{epstopdf}

\addto\captionsenglish{}

\numberwithin{equation}{section}

\newtheorem{theorem}{Theorem}[section]
\newtheorem{thm}[theorem]{Theorem}
\newtheorem{lem}[theorem]{Lemma}
\newtheorem{rem}[theorem]{Remark}
\newtheorem{cor}[theorem]{Corollary}
\newtheorem{prop}[theorem]{Proposition}

\begin{document}

\title[On  extreme values of Nehari manifold method]{On  extreme values \\ 
of Nehari manifold method \\ via nonlinear  Rayleigh's Quotient}

\author[ Yavdat Il'yasov]
{Yavdat Il'yasov} 

\address{Yavdat Il'yasov \newline
Institute of Mathematics of RAS,
Ufa, Russia}
\email{ilyasov02@gmail.com}

\thanks{The  author was 
partly supported by grants RFBR 
13-01-00294-p-a, 14-01-00736-p-a.}

\begin{abstract}

We study applicability conditions of the Nehari manifold method for the equation of the form 
$
D_u T(u)-\lambda D_u F(u)=0
$
in  a Banach space $W$, where $\lambda$ is a real parameter. Our study is based on the development of the theory Rayleigh's quotient for nonlinear problems. It turns out that the extreme values of parameter $\lambda$ for the Nehari manifold method can be found through the critical values of a corresponding  nonlinear  generalized Rayleigh's quotient. In the main part of the paper, we provide some general results on this relationship. Applications are given to several  types of nonlinear elliptic equations and systems of equations.

\end{abstract}

\maketitle

\section{Introduction}

The Nehari manifold method (NMM) \cite{Neh1, Neh} is among the most powerful and widely used tool in the analysis of equations. Let us briefly describe it.
Assume $W$ is a real Banach space, $\Phi_\lambda:W \to \mathbb{R}$ is a  Fr\'{e}chet-differentiable functional with derivative $D_u\Phi_\lambda$ and  $\lambda \in \mathbb{R}$. Consider the equation of variational form
\begin{equation}
\tag{$*$}
\label{*}
D_u\Phi_\lambda(u)=0,~~~u \in W. 
\end{equation}
 The Nehari manifold associated with  \eqref{*} is defined as the following subsets in $W$ 
$$
\mathcal{N}_\lambda:=\{u \in W\setminus 0:~D_u\Phi_\lambda(u)(u)=0\}.
$$
Since any solution of \eqref{*} belongs to $\mathcal{N}_\lambda$, a natural idea to solve \eqref{*}  is to consider the following constrained minimization problem
\begin{align*}
	&\Phi_\lambda(u) ~~\to ~~ min\\
	&D_u\Phi_\lambda(u)(u)=0,~~u \in W.
\end{align*}
 Suppose that there exists a local minimizer  $u$ of this problem and $\Phi_\lambda \in C^2(U,\mathbb{R})$ for some neighbourhood $U\subset W$ of $u$, then by the Lagrange multiplier rule
one has $ \mu_0 D_u\Phi_\lambda(u)+\mu_1 (D_u\Phi_\lambda(u)+D_{uu}\Phi_\lambda(u)(u, \cdot)) =0$ for some $\mu_0$, $\mu_1$ such that $|\mu_0|+|\mu_1| \neq 0$. Testing this equality by $u$ we obtain $\mu_1 D_{uu}\Phi_\lambda(u)(u,u)=0$. Hence, if
	$D_{uu}\Phi_\lambda(u)(u,u)\neq 0$,
then we have successively $\mu_1=0$, $\mu_0 \neq 0$ and therefore $D_u\Phi_\lambda(u)=0$. Thus, one has the following sufficient condition of the applicability of NMM
\begin{equation}
\tag{$**$}
\label{**}
	D_{uu}\Phi_\lambda(u)(u,u)\neq 0 ~~\mbox{for any}~~u \in \mathcal{N}_\lambda. 
\end{equation}
The feasibility of this condition often depends on parameter $\lambda$. Thus, we may suppose that there exists a set of the \textit{extreme values of  Nehari manifold method}  $\{\lambda_{min,i}, \lambda_{max,i}\}_{i=1}^\infty$ such that the sufficient condition \eqref{**} may hold  only when $\lambda \in \cup_i^\infty(\lambda_{min,i}, \lambda_{max,i})$. This brings up a question  of \textit{ how to find these extreme values}.

This question was a subject of investigations in \cite{ilst, ilIzv, IlyasovF}) where a  method (the so-called  spectral analysis by the fibering method \cite{Poh,PoS0,PoS1}) of the finding variational principles corresponding to the extreme values of NMM has been introduced. Although this method has been  applied to a number of  problems (see e.g. \cite{bobkov,  ilDiaz, ilCherf, ilEg,ilSari, ilegorov, ilrunst, IlRY,  ilTakac}), it has  certain disadvantages mainly due to its  complexity. The complexity becomes especially notable when we are dealing with systems of equations (see e.g. \cite{ilBobkov}).

The aim of the present paper is to introduce a new  approach to this problem.  In order to indicate the principal idea of the approach, let us consider equation \eqref{*} in the following particular form 
\begin{equation}
	D_u T(u)-\lambda D_uG(u)=0.
\end{equation}
For simplicity, we assume that  $D_uG(u)(u) \neq 0$ for any $u\in W \setminus 0$. Testing the equation by $u \in W$ and then solving  it with respect to $\lambda=:r(u)$ we obtain the following functional 
\begin{equation}\label{RQ}
r(u)=\frac{D_uT(u)(u)}{D_uG(u)(u)},~~~u\in W \setminus 0,
\end{equation}
which we call the nonlinear generalized Rayleigh quotient (NG-Rayleigh quotient for short). 
Note that  $u$ belongs $\mathcal{N}_\lambda$ if and only if it lies on the level set $r(u)=\lambda$. Using this fact we compute the following main identity 
\begin{equation}
	D_ur(u)(u) =
	\frac{1}{D_uG(u)(u)} D^2_{uu}\Phi_{\lambda}(u)(u,u), ~~\forall u\in \mathcal{N}_\lambda,
\end{equation}
which means, in particular, that the sufficient condition \eqref{**} holds if and only if $D_ur(u)(u)\neq 0$. Notice that  $D_ur(u)(u)=\partial r(tu)/\partial t|_{t=1}$.
These reasonings lead us to the main idea of our approach: 


\textit{The extreme values of the Nehari manifold method  can be found by means of studying  the critical values of the fibered NG-Rayleigh quotient  $\tilde{r}(t,u):=r(tu)$, $t\in \mathbb{R}^+$, $u \in W \setminus 0$ }. 

This idea  is consistent with the variational method of the finding eigenvalues of linear operators which has been introduced  in 1869 by Weber \cite{weber}  and then developed in works by Rayleigh, Fisher, Ritz, Courant (see e.g. \cite{courant,fisher,  Ritz, Rayleigh}). Indeed,  by the minimax theorem the set of eigenvalues $\sigma:=\{\lambda_1,...,\lambda_n\}$ of the $n \times n$ Hermitian matrix $A$ corresponds to the set of the critical values of Rayleigh's quotient $r(u):=\frac{\left\langle Au,u \right\rangle  }{\left\langle u,u \right\rangle }$. Furthermore, if we formally consider  
 the Nehari manifold minimization problem 
\begin{align}
\begin{cases}\label{NW}
		&\left\langle Au,u \right\rangle-\lambda \left\langle u,u \right\rangle ~~\to ~~ min \\
	&\left\langle Au,u \right\rangle-\lambda \left\langle u,u \right\rangle=0,~~u \in \mathbb{R}^n \setminus 0,
		\end{cases}
\end{align}
then, as above, to apply it we should verify  whether a priori solution of \eqref{NW} will satisfy the equation $Au-\lambda u=0$. Thereby, we need to find the corresponding  set of extreme values of NMM. The application of the above idea yields that this set coincides with  the critical values of Rayleigh's quotient $r(u)$. Thus, in the linear case of \eqref{*}, the problem of the finding extreme values of NMM is nothing else than the finding spectrum $\{\lambda_1,...,\lambda_n\}$ of $A$, i.e. this is the eigenvalues problem.

The paper is organized as follows. Section 2 contains some preliminaries on the Nehari manifold method. Note that if \eqref{*} is a system of equations there are several ways of introducing of the Nehari manifold. We discuss, in particular, the so-called  vector and scalar Nehari manifold methods. Section 3 is devoted to the nonlinear generalized Rayleigh's quotient and its main properties.  In Section 4, we introduce some basic extreme values of NMM by studying the critical values of the following functionals
$\lambda(u):=\inf_{t>0}\, r(tu)$, $\Lambda(u):=\sup_{t>0}\, r(tu)~~ \mbox{on}~~W\setminus 0$.
In Section 5, we present several applications of the method  where the extreme values of NMM can be expressed  in explicit  variational form. 
  The aim of Section 7 is to show that NG-Rayleigh's quotient can be a useful tool  itself in the analysis of equations. In particular, we prove using NG-Rayleigh's quotient a result on the existence of multiple solutions for an abstract equation and then, as a consequence, we obtain a novel result on the existence of multiple sign-constant solutions for a boundary value problem with a general convex-concave type nonlinearity and $p$-Laplacian.

\medskip

\par\noindent
{\bf Notations}

\noindent

We will denote by $W=W_1\times ...\times W_n$ the product of real Banach
spaces $W_i$ with the norms $||\cdot||_{W_i}$, $i=1,...,n$ and the norm  $||\cdot||=||\cdot||_{W_1}+...+||\cdot||_{W_n}$ in  $W$. 
 
To simplify the notation we write: 
\begin{itemize}
\item $\dot{W}=(W_1\setminus 0)\times ...\times (W_n\setminus 0)$ and  $\dot{\mathbb{R}}^+=\mathbb{R}^+\setminus 0$,
	\item $t:=(t_1,...,t_n) \in \mathbb{R}^n$,
	\item $t\cdot u:=(t_1u_1,...,t_nu_n)$, $tu:=(tu_1,...,tu_n)$, $\left\langle t,u\right\rangle=\sum_{i=1}^n t_iu_i$, for $u\in W$, $t \in \mathbb{R}^n$,
	\item $1_n=(1,...,1)^T$ and $0_n=(0,...,0)^T $ denote the vectors $1\times n$ and $0\times n$ in $\mathbb{R}^n$, respectively.
	\end{itemize}
\par\noindent
For $F \in C^1(W,\mathbb{R})$, $u \in W$, we write 
\begin{itemize}
	\item $\nabla_u F(u):=(D_{u_1}F(u),...,D_{u_n}F(u))^T$,
	\item $\nabla_u F(u)(v):=(D_{u_1}F(u)(v_1),...,D_{u_n}F(u)(v_n))^T,~~v \in W$, \\
	\item $D_u F(u)(v):= \sum_{i=1}^n D_{u_i}F(u)(v_i)$,
\end{itemize}
 where $D_{u_i}F(u)$ is the Frechet  derivative with respect to $u_i \in W_i$ and   $D_{u_i}F(u)(v_i)$ denotes the evaluation of $D_{u_i}F(u)$ at $v_i \in W_i$,  $i=1,2,...,n$. 

 \par\noindent
For a given $F: W \to  \mathbb{R}$, by the  \textit{(vector) fibered map} $\tilde{F}:(\mathbb{R}^+)^n\times W \to  \mathbb{R}$ and the\textit{ scalar fibered map} $\tilde{F}^{sc}:\mathbb{R}^+\times W \to  \mathbb{R}$ we mean the maps which are defined by $\tilde{F}(t,u):=F(t\cdot u)$ for $(t,u) \in (\mathbb{R}^+)^n\times W$ and 
 $\tilde{F}^{sc}(t,u):=F(tu)$ for $(t,u) \in \mathbb{R}^+\times W$, respectively. We write
  \begin{itemize}
	\item $\nabla_t F(t\cdot u):=(\partial_{t_1}F(t\cdot u),...,\partial_{t_n}F(t\cdot u) )^T$,
	\item $\nabla_t F(t\cdot u)(t):=(\partial_{t_1}F(t\cdot u)t_1,...,\partial_{t_n}F(t\cdot u)t_n )^T,$
	\item $\partial F(t\cdot u)/\partial t:=\left\langle \nabla_t F(t\cdot u), t\right\rangle\equiv \sum_{i=1}^n \partial_{t_i}F(t\cdot u) t_i$.
 \end{itemize}

\section{Preliminaries}

In the present paper, we shall deal with the $n$-dimenosional system of  equations of the form
\medskip
\begin{equation}\label{V1}
	\nabla_u \Phi_\lambda(u)\equiv \nabla_u T(u)-\lambda  \nabla_u G(u)=0,~~u \in \dot{W},
\end{equation}
where $\dot{W}=\Pi_{i=1}^n(W_i\setminus 0)$,  $T,G \in C^1(\dot{W}, \mathbb{R})$, $\lambda \in \mathbb{R}$ and $\Phi_\lambda(u)=T(u)-\lambda G(u)$. 
In the case  $n=1$, we call \eqref{V1} the \textit{scalar problem}. 
We define the Nehari manifold associated with \eqref{V1} as follows: 
\begin{equation}\label{NehMan}
	\mathcal{N}_\lambda=\{u\in \dot{W}:~ \nabla_u \Phi_\lambda(u)(u)\equiv \nabla_t \Phi_\lambda(t\cdot u)|_{t=1_n}=0\}.
\end{equation}
Then the corresponding Nehari manifold minimization problem is 
\begin{align}
		\begin{cases}\label{N11}
	&\Phi_\lambda(u) \to \textsl{min}\\	
	& \nabla_u \Phi_\lambda(u)(u)=0,~~u\in \dot{W}.
	\end{cases}
	\end{align}
	
We will say that $u_0 \in  \mathcal{N}_\lambda$ is a solution (or local minimizer) of \eqref{N11} if there exists $\delta>0$ such that 
$
\Phi_\lambda(u_0)\leq \Phi_\lambda(u_0)~~~\mbox{whenever}~~~||u-u_0||_W<\delta, ~~u \in  \mathcal{N}_\lambda,
$
and we denote by $\hat{\Phi}_\lambda$ the global minimization value in \eqref{N11}, i.e.  
$\hat{\Phi}_\lambda:=  \inf	\{\Phi_\lambda(u):~u \in \mathcal{N}_\lambda\}. $
A solution $u \in \dot{W}$ of \eqref{V1} is said to be ground state if  there holds
$
\Phi_\lambda(u)\leq \Phi_\lambda(w)~
$
for any solution $w \in \dot{W}$ of \eqref{V1}. Thus a global minimizer $u$ of \eqref{N11} which satisfies equation \eqref{V1} is a ground state.

From now on we make the following assumption: 

\medskip 
\noindent
 \textit{$\nabla_t \Phi_\lambda(t\cdot u)$ is a map of class $C^1$ on $(\dot{\mathbb{R}}^+)^n \times \dot{W}$, i.e.
$\nabla_t \tilde{\Phi}_\lambda \in C^1((\dot{\mathbb{R}}^+)^n \times \dot{W}, \mathbb{R}^n)$.}

\medskip 
\noindent
Notice this assumption implies that the constraints $\nabla_u \Phi_\lambda(u)(u)$ in \eqref{N11} is a manifolds of class $C^1$ on $\dot{W}$.  Obviously  any functional $\Phi_\lambda \in C^2(\dot{W}, \mathbb{R})$ satisfies to this assumption. 


Let  $u \in  \dot{W}$. Consider the Jacobian matrix of vector-valued function $\Psi(t):=\nabla_u \Phi_\lambda(t\cdot u)(t\cdot u)$, i.e. 
$$
J_t\left(\nabla_u \Phi_\lambda(t\cdot u)(t\cdot u)\right)=[\frac{\partial}{\partial t_1}\Psi(t)~...~\frac{\partial}{\partial t_n} \Psi(t)]
$$
or, component-wise:
$$
J_t\left(\nabla_u \Phi_\lambda(t\cdot u)(t\cdot u)\right)_{i,j}=\frac{\partial^2}{\partial t_i \partial t_j}\Phi_\lambda(t\cdot u)t_j+\delta_{ij}\frac{\partial }{\partial t_j }\Phi_\lambda(t\cdot u),
$$
where $\delta_{ij}=1$ if $i=j$ and $\delta_{ij}=0$ if $i\neq j$.
To shorten the notation, we write $J(\nabla_u \Phi_\lambda(u)(u)):=J_t(\nabla_u \Phi_\lambda(t\cdot u)(t\cdot u))|_{t=1_n}$. Note that  in the case $\Phi_\lambda \in C^2(\dot{W}, \mathbb{R})$ one has 
$$
J(\nabla_u \Phi_\lambda(u)(u))=\left(D_{u_i u_j}^2\Phi_\lambda( u)(u_i,u_j)+\delta_{ij} D_{u_j}\Phi_\lambda(u)(u_j)\right)_{\{1\leq i,j\leq n\}}.
$$
Furthermore, if $u \in \mathcal{N}_\lambda$, then 
\begin{equation}\label{JN}
	J(\nabla_u \Phi_\lambda(u)(u))=\left(D_{u_i u_j}^2\Phi_\lambda( u)(u_i,u_j)\right)_{\{1\leq i,j\leq n\}}.
\end{equation}
Let us prove
\begin{lem}\label{lem1}
Assume $\Phi_\lambda \in  C^1(\dot{W}, \mathbb{R})$, $\nabla_t \tilde{\Phi}_\lambda \in C^1((\dot{\mathbb{R}}^+)^n \times \dot{W}, \mathbb{R}^n)$.  Suppose that there exists a solution  $u_0$ of problem \eqref{N11} such that  
\begin{equation} \label{Cond}
	\det	J(\nabla_u \Phi_\lambda(u_0)(u_0))\neq 0.
\end{equation}
 Then $u_0$ satisfies  equation \eqref{V1}.
\end{lem}
\begin{proof}  To simplify notation, we prove under the assumption $\Phi_\lambda \in C^2(\dot{W}, \mathbb{R})$. By Fritz John \cite{john} conditions there exist the Lagrange multipliers $\mu_0$, $\mu_1,...,\mu_n$ such that $\sum_{i=0}^n|\mu_i|\neq 0$ and 
$$
\mu_0 D_{u_i}\Phi_\lambda(u_0)+\sum_{j=1}^{n}\mu_j(D_{u_iu_j}^2\Phi_\lambda(u_0)(u_{0,j}, \cdot)+\delta_{ij}D_{u_i}\Phi_\lambda(u_0)) =0,~~i=1,...,n.
$$
Test these equations by $u_{0,i}$, $i=1,...,n$, respectively. Then since  $u_0 \in \mathcal{N}_\lambda$, we obtain 
$ J_t(\nabla_u \Phi_\lambda(u)(u))\bar{\mu}=0$, where  $\bar{\mu}:=(\mu_1,...,\mu_n)^T$.
However, in view of \eqref{Cond}, this is possible only if $\mu_1=0,...,\mu_n=0$. Hence, $\mu_0 \neq 0$ and we obtain the required. 
\end{proof}

Let us mention that the Nehari manifold \eqref{NehMan} is actually introduced   by means of the vector fibered map $\tilde{\Phi}_\lambda(t,u):=\Phi_\lambda(t\cdot u)$,  $(t,u) \in (\mathbb{R}^+)^n\times W$. However, there is another approach which is based on the scalar fibered map $\tilde{\Phi}^{sc}_\lambda(t,u):=\Phi_\lambda(tu)$, $(t,u) \in \mathbb{R}^+\times W$. Indeed, let us introduce  the \textit{scalar Nehari manifold}
\begin{equation}\label{NehSc}
\mathcal{N}^{sc}_\lambda=\{u\in W\setminus 0_n:~ \partial_t \Phi_\lambda(tu)|_{t=1}\equiv D_{u}\Phi_\lambda(u)(u)\equiv \sum_{i=1}^n D_{u_i}\Phi_\lambda(u)(u_i)=0\}.
\end{equation}
Then the corresponding \textit{scalar Nehari manifold minimization problem } is defined as follows
\begin{align}
		\begin{cases}\label{N11KW}
	&\Phi_\lambda(u) \to min\\	
	& D_{u}\Phi_\lambda(u)(u)\equiv \sum_{i=1}^n D_{u_i}\Phi_\lambda(u)(u_i)=0,~~u\in W\setminus 0_n.
	\end{cases}
	\end{align}
Arguing as above, we have 
\begin{lem}\label{lem1KW}
Assume $\Phi_\lambda \in  C^1(W\setminus 0_n, \mathbb{R})$, $\partial_t \tilde{\Phi}_\lambda \in C^1(\dot{\mathbb{R}}^+\times (W\setminus 0_n), \mathbb{R})$.  Suppose that there exists a solution  $u$ of problem \eqref{N11KW} such that  
\begin{equation} \label{CondKW}
	\frac{\partial}{\partial t  }D_{u}\Phi_\lambda(tu)(tu)|_{t=1}\neq 0.
\end{equation}
 Then $u$ satisfies  equation \eqref{V1}.
\end{lem}
Note that in the case $\Phi_\lambda \in C^2(W\setminus 0_n, \mathbb{R})$ the condition \eqref{CondKW} may be written as  
$$
D^2_{uu}\Phi_\lambda(u)(u,u)=\sum_{i,j=1}^n ( D^2_{u_iu_j}\Phi_\lambda(u)(u_i,u_j)+\delta_{ij}D_{u_i}\Phi_\lambda(u)(u))\neq 0.
$$
Let us remark that the assertion of Lemma \ref{lem1KW} follows directly from Lemma \ref{lem1} if we formally regard the system of equation \eqref{V1}  as a scalar equation.  

In what follows, we call \eqref{N11} and \eqref{N11KW} the vector and scalar Nehari manifold method (NMM), respectively. Nevertheless,  in most cases, we call \eqref{N11} simply Nehari manifold method when no confusion can appears.

\begin{rem}\label{VSc}
	In the sequel, it will cause no confusion if we consider the method \eqref{N11KW} as a particular case of the general approach \eqref{N11}.   Moreover in the future mainly our reasoning (unless otherwise stated) will be carried out on problem \eqref{N11}, meaning that they are also valid for \eqref{N11KW} as for the special case of \eqref{N11}.
\end{rem}

\begin{rem}\label{RScV}
	It is worth noticing that $\nabla_u \Phi_\lambda(u)(u)=0$ and $\det	J(\nabla_u \Phi_\lambda(u)(u))=0$ for any $u \in W\setminus \dot{W}$, whereas $D_{u}\Phi_\lambda(u)(u)=0$ and $\frac{\partial^2}{\partial t^2  }\Phi_\lambda(t u)|_{t=1}=0$ if $u=0_n$. This is why in definitions \eqref{NehMan}, \eqref{NehSc}, we are setting	$\mathcal{N}_\lambda\subset \dot{W}$, whereas $\mathcal{N}^{sc}_\lambda\subset  W\setminus 0_n$.  
\end{rem}

In the literature, the Nehari manifold minimization problem  is sometimes considered in the following  form (see e.g. \cite{DrabPoh, ilIzv, Poh, PoS0, PoS1})
\begin{align}
		\begin{cases}\label{N11v}
	&\Phi_\lambda(t\cdot v) \to min\\	
	& \nabla_t\Phi_\lambda(t\cdot v)=0_n\\
	& v \in S, ~t_i>0,~~i=1,2,...,n,
	\end{cases}
	\end{align}
where $S:=\{v \in W:~||v||_W=1\}$ and the Nehari manifold is defined as follows
\begin{equation}\label{NehMan2}
	\tilde{\mathcal{N}}_\lambda=\{(t,v)\in  (\dot{\mathbb{R}}^+)^n\times S:~\nabla_t\Phi_\lambda(t\cdot v)=0_n
	\}.
\end{equation}
It is readily seen that if $u$ is a solution of problem \eqref{N11} such that  
\eqref{Cond} is satisfied, then $(t,v)$, where $t=1_n ||u||$, $v=u/||u||$, is also a solution of \eqref{N11v} such that 
\begin{equation}\label{Cond1}
\det(J_t(\nabla_t\Phi_\lambda(t \cdot v)))\neq 0.	
\end{equation}
The converse is also true.

\begin{prop}\label{prop2}
Assume  $\Phi_\lambda \in  C^1(\dot{W}, \mathbb{R})$, $\nabla_t \tilde{\Phi}_\lambda \in C^1((\dot{\mathbb{R}}^+)^n \times \dot{W}, \mathbb{R}^n)$. Suppose $(t_0,v_0) \in \tilde{\mathcal{N}}_\lambda$ such that \eqref{Cond1} is satisfied. Then  there
exists a neighbourhood $U(v_0)$ of $v_0$ in $S$ and a unique local $C^1$-map
$
t:\:U(v_0) \to \mathbb{R}^n
$
such that $(t(v),v) \in \tilde{\mathcal{N}}_\lambda$ for all $v \in
U(v_0)$ and $t(v_0)=t_0$. 
\end{prop}
\begin{proof}
By the assumption the determinant of the Jacobian  of $\Psi(t):=\nabla_t \Phi_\lambda(t\cdot v_0)$ at  the point $t=t_0$  is nonzero. Thus we may apply the Implicit
Function Theorem to the functional $\psi(t,v):=\nabla_t \Phi_\lambda(t\cdot v)$ and the proof  follows.
\end{proof}
From this we are able to prove the following analogue of Lemma \ref{lem1}
\begin{lem}\label{lem1f}
	Assume  $\Phi_\lambda \in  C^1(\dot{W}, \mathbb{R})$, $\nabla_t \tilde{\Phi}_\lambda \in C^1((\dot{\mathbb{R}}^+)^n \times \dot{W}, \mathbb{R}^n)$.  Suppose that there exists a solution  $(t_0,v_0)$ of problem \eqref{N11v} such that \eqref{Cond1} holds.
Then $u_0=t_0\cdot v_0$ satisfies  equation \eqref{V1}.
\end{lem}
\begin{proof}
Since \eqref{Cond1} holds, we may apply Proposition \ref{prop2}. Thus  there
exists a neighbourhood $U(v_0) \subset S$ and a unique local $C^1$-map $t:\:U(v_0) \to \mathbb{R}^n$ such that $(t(v),v) \in \tilde{\mathcal{N}}_\lambda$ for all $v \in
U(v_0)$ and $t(v_0)=t_0$. This implies that the function $\Phi_\lambda(t(v)v)$ constrained on $S$ attains a local minimum at point $v_0 \in U(v_0)$. Consequently, by Fritz John conditions there exist the Lagrange multipliers $\mu_0$, $\mu_1$ such that $|\mu_0|+|\mu_1|\neq 0$ and 
$$
\mu_0 \nabla_v \Phi_\lambda(t(v_0)\cdot v_0)+\mu_1 \nabla_v ||v_0||=0.
$$
Now testing this system of equations by $v_0$ we obtain as above in the proof of Lemma \ref{lem1} that
$\mu_0 D_v \Phi_\lambda(t_0\cdot v_0)(v_0)+\mu_1 =0$. Since $\nabla_t\Phi_\lambda(t_0\cdot v_0)=0_n$ we have $ D_v \Phi_\lambda(t_0\cdot v_0)(v_0)=0$ and consequently $\mu_1=0$. This concludes the proof.
\end{proof}

It is important to note that the definition of the  Nehari manifold \eqref{NehMan} \rm{(\eqref{NehMan2})} and condition  \eqref{Cond} \rm{(\eqref{Cond1})}  are invariant in the following sense

\begin{prop}\label{propInvar}
 Let $\psi:(\dot{\mathbb{R}}^+)^n \to(\dot{\mathbb{R}}^+)^n$ be $C^1$-map such that $\psi(1_n)=1_n$, $det(J(\psi(s))|_{s=1_n}\neq 0$. Then   $\nabla_s \Phi_\lambda(\psi(s)\cdot u)|_{s=1_n}=0$ if and only if  $\nabla_t \Phi_\lambda(t\cdot u)|_{t=1_n}=0$, and 
  $\det J_t(\nabla_s \Phi_\lambda(\psi(s)\cdot u)(\psi(s)\cdot u))|_{s=1_n} \neq 0$ if and only if  $det J(\nabla_u \Phi_\lambda(u)(u))\neq 0$. 
\end{prop}
\begin{proof}
It follows directly since	$\nabla_s \Phi_\lambda(\psi(s)\cdot u)|_{s=1_n}=
J(\psi(s))|_{s=1_n}\nabla_t \Phi_\lambda(t\cdot u)|_{t=1_n}$ and $\det J(\nabla_s \Phi_\lambda(\psi(s)\cdot u)(\psi(s)\cdot u))|_{s=1_n}$ =$ \det J(\psi(1_n)) $$ \det J(\nabla_u \Phi_\lambda(u)(u))$.
\end{proof}

\medskip
\section{Nonlinear generalized Rayleigh's quotient}

 Denote $\mathcal{W}:=\{u\in W: ~D_uG(u)(u)\neq 0\}$. Evidently $0_n \not\in \mathcal{W}$ and  $\mathcal{W}$ is an open subset in $W$ if $G \in C^1(W, \mathbb{R})$. In the sequel, we always assume that $t\cdot u \in \mathcal{W}$ for any $u\in \mathcal{W}$ and $t \in \dot{\mathbb{R}}^n$.

The following functional plays a fundamental role in the present paper:
 \begin{equation}\label{lamb}
	r(u)=\frac{D_uT(u)(u)}{D_uG(u)(u)},~~~u\in\mathcal{W}.
\end{equation}
We stress that $r(u)$ has been obtained by the following rule. Consider system of equations \eqref{V1}.
Test these equations by $u_{i}$, $i=1,...,n$, respectively. Summation then yields
$D_uT(u)(u)-\lambda D_uG(u)(u)=0$. Now solving  this equation with respect to $\lambda$ we obtain   the functional \eqref{lamb}. Notice that in the  linear case of \eqref{V1} $Au-\lambda u=0$, where $A$ is a $n \times n$ Hermitian matrix, 
 the same rule gives the function
 $$
r(u)=\frac{\left\langle Au,u \right\rangle  }{\left\langle u,u \right\rangle }
$$ 
which is nothing else than ordinary Rayleigh's quotient \cite{Rayleigh, weber}. For this reason, it makes sense to call  \eqref{lamb} the \textit{nonlinear generalized Rayleigh's quotient} (  \textit{NG-Rayleigh's quotient} for short).

Note that $T,G \in C^1(W, \mathbb{R})$ implies $r(\cdot) \in C^1(\mathcal{W}, \mathbb{R})$.
In what follows, we call $\tilde{r}(t,u):=r(t\cdot u)$ defined on $(\dot{\mathbb{R}}^+)^n\times \mathcal{W}$ a \textit{vector fibered NG-Rayleigh's quotient} and $\tilde{r}^{sc}(t,u):=r(t u)$ defined on $\dot{\mathbb{R}}^+\times \mathcal{W}$ a \textit{scalar fibered NG-Rayleigh's quotient}. For shorten notation, we will call $\tilde{r}(t,u)$ simply  fibered NG-Rayleigh's quotient when no confusion can appears.   Clearly, $\tilde{r}(\cdot,u) \in C^1((\dot{\mathbb{R}}^+)^n, \mathbb{R})$ and $\tilde{r}^{sc}(\cdot,u) \in C^1(\dot{\mathbb{R}}^+, \mathbb{R})$ for every $u\in\mathcal{W}$.

Our basic assumption on NG-Rayleigh's quotient $r(u)$ is the following: For every fixed $u\in\mathcal{W}$ and $a_n \in (\dot{\mathbb{R}}^+)^n \setminus \dot{\mathbb{R}}^n$, there exists  $\lim_{t\to a_n}r(t\cdot u)=\hat{r}(a_n,u)$, where $|\hat{r}(a_n,u)|\leq \infty$, thereby there exists a continuation of the fibered mapping  $\tilde{r}(t,u):=r(t\cdot u)$ to $(\dot{\mathbb{R}}^+)^n\times \mathcal{W} $ by setting $\tilde{r}(a_n,u):=\hat{r}(a_n,u)$ for each $u \in \mathcal{W}$ and $a_n \in (\dot{\mathbb{R}}^+)^n \setminus \dot{\mathbb{R}}^n$. In the scalar case, we make the similar assumption: For every fixed $u\in \mathcal{W}$, there exists $\lim_{t\to 0}\tilde{r}^{sc}(t,u)=\tilde{r}^{sc}(0,u)$, thereby there exists a continuation of  $\tilde{r}^{sc}(t,u)$ on $\mathbb{R}^+\times \mathcal{W}$.

 In the next corollaries we collected some basic  properties of $r(u)$.

\begin{cor} \label{LP0}
	For any $u \in\mathcal{W}$ and $t \in (\dot{\mathbb{R}}^+)^n$ there holds 
	
		(i)\, $r(t\cdot u) =\lambda$ if and only if $\partial \Phi_\lambda(t\cdot u)/\partial t= 0$;
		\par
		(ii)\, if $t\cdot u \in \mathcal{N}_{\lambda}$, then $\lambda=r(t\cdot u)$.
	
		Furthermore, if $D_u G(t\cdot u)(t\cdot u) > 0$ $(D_u G(t\cdot u)(t\cdot u) < 0)$  for  $u \in \mathcal{W}$ and $t\in (\dot{\mathbb{R}}^+)^n$, then:
		\par
	(iii) \, $r(t\cdot u) >  \lambda$  if and only if $\partial \Phi_\lambda(t\cdot u)/\partial t> 0$;\par
		(iv)\, $r(t\cdot u) <  \lambda$ if and only if $\partial \Phi_\lambda(tu)/\partial t< 0$.

\end{cor}
\begin{proof} To obtain the proof it is sufficient to note that $\lambda=r(t\cdot u)$ is nothing else than the root of the equation  
$$
\partial \Phi_{\lambda}(t\cdot u)/\partial t\equiv \left\langle  \nabla_t \Phi_{\lambda}(t\cdot u), t\right\rangle= \left\langle \nabla_t T(t\cdot u), t\right\rangle-\lambda  \, \left\langle \nabla_t G(t\cdot u),  t\right\rangle=0.
$$
\end{proof}
For the scalar fibered NG-Rayleigh's quotient, in addition to Corollary \ref{LP0}  we have  
\begin{cor} \label{LP}
	Consider the scalar Nehari manifold \eqref{NehSc}. Then for $u \in \mathcal{W}$ and $t>0$ there holds:\par
(i) \, $tu \in \mathcal{N}^{sc}_{\lambda}$ if and only if $\lambda=r(tu)$.		
	
	Furthermore, if $D_u G(tu)(tu) > 0$ $(D_u G(tu)(tu) < 0)$  for  $u \in \mathcal{W}$ and $t\in \mathbb{R}^+$, then:\par 
		(ii)\, $\partial r(tu)/\partial t<0$   if and only if $\partial^2 \Phi_\lambda(tu)/\partial t^2<0$ ($\partial^2 \Phi_\lambda(tu)/\partial t^2>0$);\par
		(iii)\, $\partial r(tu)/\partial t>0$  if and only if $\partial^2 \Phi_\lambda(tu)/\partial t^2>0$ ($\partial^2 \Phi_\lambda(tu)/\partial t^2<0$).
			\end{cor}
{\it Proof\,} is evident.

\medskip

Note that in the scalar case if $\mathcal{W}=W\setminus 0$, then  $u \in \mathcal{N}^{sc}_{\lambda}$ if and only if $\lambda=r(u)$. Therefore in this case the Nehari manifold  can be defined also as follows
\begin{equation}\label{NehL}
	\mathcal{N}^{sc}_\lambda=\{u \in W\setminus 0:~r(u)=\lambda\},~~\lambda \in \mathbb{R},
\end{equation}
and the minimization problem \eqref{N11} for $\lambda \in \mathbb{R}$ can be written
in the following equivalent form
\begin{align}
		\begin{cases}\label{N12}
	&\Phi_\lambda(u) \to min\\	
	& r(u) =\lambda,~~u \in W\setminus 0.
	\end{cases}
	\end{align}

Let $u\in \mathcal{W}$. We call $t_u \in (\dot{\mathbb{R}}^+)^n $ the \textit{extremal point} of $r(t\cdot u)$ if the function $r(t\cdot u)$ attains  at $t_u$ its local maximum or minimum on $(\dot{\mathbb{R}}^+)^n$. If $\nabla_t r(t_u \cdot u)=0_n$, then $t_u$ is said to be a \textit{critical point} of $r(t\cdot u)$ and  $\lambda=r(t_u \cdot u)$ is called the \textit{critical value}.

Let $u \in \mathcal{W}$ and  $t\in (\dot{\mathbb{R}}^+)^n$.  Then by \eqref{lamb} we have
\begin{equation}\label{grad}
	\nabla_t r(t\cdot u) =
	\frac{J_t(\nabla_u \Phi_\lambda(t\cdot u)(t \cdot u))1_n|_{\lambda=r(t\cdot u)}}{D_uG(t\cdot u)(t\cdot u)} .
\end{equation}

\begin{prop}\label{grad1}
For any $u\in \mathcal{W}$ and  $t\in (\dot{\mathbb{R}}^+)^n$	such that $t\cdot u \in \mathcal{N}_\lambda$ with $\lambda=r(t\cdot u)$,  there holds
\par
(i)\,  $\nabla_tr(t\cdot u)=0_n$ implies $ \det	J_t(\nabla_u \Phi_\lambda(t\cdot u)(t \cdot u))= 0$;
	\par (ii)\, $J_t(\nabla_u \Phi_\lambda(t\cdot u)(t \cdot u))1_n=0_n$  implies $\nabla_t r(t\cdot u)=0$.

\end{prop}
{\it Proof\,} is evident.

\medskip

Observe, in  the  case of scalar NMM, since $D_{u}\Phi_\lambda(t u)(u)|_{\lambda=r(tu)}= 0$ for any $u \in \mathcal{W}$, formula \eqref{grad} can be written as follows
\begin{eqnarray}\label{grad1a}
	\partial_t r(tu) =
	\frac{1}{D_uG(tu)(tu)} D^2_{uu}\Phi_{\lambda}(tu)(tu,u)|_{\lambda=r(tu)}.
\end{eqnarray}
Hence, in this case,  $\partial_t r(tu) =0$ for $t>0$ if and only if $D^2_{uu}\Phi_{\lambda}(tu)(tu,u)=0$, $\lambda=r(tu)$.

We will make the following assumption:
\medskip
	\begin{quote}
		{\bf (A)}\, {\it If $u\in \mathcal{W}$, $t\in (\dot{\mathbb{R}}^+)$ and $t\cdot u \in \mathcal{N}_{r(t\cdot u)}$, then the equality  
			$$ 
		\det	J_t(\nabla_u \Phi_\lambda(t\cdot u)(t \cdot u))|_{\lambda=r(t\cdot u)}= 0
		$$ 
		implies that $t$ is an extremal point of $r(t\cdot u)$.}
\end{quote}
		
\medskip
\begin{rem}\label{RemScal1}
	In  the  case of scalar NMM, in view of \eqref{grad1a}  assumption {\rm \bf (A)} means that $r(tu)$ for  $u \in \mathcal{W}$  has no critical points in $\mathbb{R}\setminus 0$  other than extremal. 
\end{rem}

In general position, one may assume  that for every $u \in \mathcal{W}$ the function  $r(t\cdot u)$ has a countable (or finite) set of extremal points $t^1_u$, $t^2_u$ ..., which determine the continuous mappings $t^i_{(\cdot)}:  \dot{W} \to (\dot{\mathbb{R}}^+)^n$, $i=1,2,...$ so that there exist the functionals $\lambda_i(u):=r(t^i_u\cdot u)$, $u \in \dot{W}$, $i=1,2,...$. Consider 
$$
\lambda_{min,i}=\inf_{u \in \mathcal{W}}\lambda_i(u),~~\lambda_{max,i}=\sup_{u \in \mathcal{W}}\lambda_i(u), ~~i=1,2,... .
$$ 
Our basic  conjecture is that the set of  extreme values of Nehari manifold method contains in the set $\sigma:=\{\lambda_{min,i},\lambda_{max,i}\}_{i=1}^\infty$.

Let us stress that $\lambda_{min,i}$ and $\lambda_{max,i}$ are 0-homogeneous functionals on $\mathcal{W}$, i.e. $\lambda_{min,i}(su)=\lambda_{min,i}(u)$, $\lambda_{max,i}(su)=\lambda_{max,i}(u)$ for any $s \in  (\dot{\mathbb{R}}^+ )^n$ and $u\in \mathcal{W}$. Note that ordinary Rayleigh's quotient $r(u)=\frac{\left\langle Au,u \right\rangle  }{\left\langle u,u \right\rangle }$ possess the similar property.

In the present paper,  to find extreme values of NMM, we mainly deal with  the following functionals
$$
\lambda(u):=\inf_{t\in  (\mathbb{R}^+)^n}\, r(t\cdot u),~~~~\Lambda(u):=\sup_{t\in  (\mathbb{R}^+)^n}\, r(t\cdot u)~~ \mbox{on}~~\mathcal{W}.
$$ 
 Notice that if there exist  the extremal points $t_{u,max}$, $t_{u,min}$, where the function $\tilde{r}(\cdot, u)$ attains its  global maximum and minimum values, respectively, then 
 $\lambda(u)= r(t_{u,min}\cdot u)$, $\Lambda(u)=r(t_{u,max}\cdot u)$. 

\begin{rem} \label{RemInv}
	In view of Proposition \ref{propInvar}, all of the above  statements (Corollaries \ref{LP0}, \ref{LP}, Proposition \ref{grad1} etc.) still hold after making a change of variable $t=\psi(s)$, where $\psi:(\dot{\mathbb{R}}^+)^n \to(\dot{\mathbb{R}}^+)^n$ is a $C^1$-map such that  $det(J(\psi(s))\neq 0$ for all $s\in (\dot{\mathbb{R}}^+)^n$. Furthermore, {\bf (A)} is satisfied if and only if the same assumption {\bf (A)} holds  after making a change of variable $t=\psi(s)$.
\end{rem}

\section{Some basic extreme values} 
 In this section, using  $\lambda(u),\Lambda(u)$ we introduce  some
extreme values of  the Nehari manifold method  that we believe are common for most problems.

Introduce
\begin{align}
	\lambda_{min}&=\inf_{u \in \mathcal{W}}\,\lambda(u)\equiv \inf_{u \in \mathcal{W}}\,\inf_{t\in  (\mathbb{R}^+)^n}\, r(t\cdot u), \label{lammi1}\\
	\lambda_{max}&=\sup_{u \in \mathcal{W}}\,\Lambda(u)\equiv \sup_{u \in \mathcal{W}}\,\sup_{t\in  (\mathbb{R}^+)^n}\, r(t\cdot u). \label{lammax1}
\end{align}
Similarly, define $\lambda^{sc}_{min}$, $\lambda^{sc}_{max}$ for the scalar fibered NG-Rayleigh's quotient $r(tu)$ instead of $r(t\cdot u)$. 

\begin{lem}\label{lem0}
\par\noindent
\begin{enumerate}
		\item If $\lambda_{min}>-\infty$ $(\lambda_{max}<+\infty)$, then $\mathcal{N}_{\lambda}=\emptyset$ for any 
$\lambda<\lambda_{min}$ $(\lambda>\lambda_{max})$. 
	\item $\mathcal{N}^{sc}_{\lambda}\neq \emptyset$ for any $\lambda 
	\in (\lambda^{sc}_{min},\lambda^{sc}_{max})$.
\end{enumerate}
\end{lem}
\begin{proof}
(1) follows immediately,	since for $u\in \mathcal{N}_{\lambda}$, we have  $r(u)=\lambda$. (2) holds, because by Corollary \ref{LP0}, $u\in \mathcal{N}^{sc}_{\lambda}$ if and only if $r(u)=\lambda$.
\end{proof}

Note that since any solution of \eqref{V1} belongs  $\mathcal{N}_{\lambda}$, the values $\lambda_{min}$, $\lambda_{max}$ set bound  the existence of any solution of \eqref{V1}, that is for all  $\lambda<\lambda_{min}$ and $\lambda>\lambda_{max}$ (provided $\lambda_{min}>-\infty$, $\lambda_{max}<+\infty$) equation \eqref{V1} has no solutions in $W$.

 The next assumption is technical. It specifies the shape of $r(t\cdot u)$ for which we prove our main results.
\medskip
\par\noindent
\begin{quote}
	{\bf (S)}\, \,{\it
			For any $u\in \mathcal{W}$ and  $t\in (\dot{\mathbb{R}}^+)^n$, s.t. $t\cdot u \in \mathcal{N}_{r(t\cdot u)}$, the condition $\nabla_tr(t\cdot u)= 0_n$ implies that the function $r(t\cdot u)$  attains its global minimum or/and maximum at the point $t$ in $(\dot{\mathbb{R}}^+)^n$}. 
\end{quote}
		
\medskip

\begin{rem}\label{R1}
	In other words, condition {\bf (S)} means that for every $u \in \mathcal{W}$ one of the following holds: (i) $r(t \cdot u)$ has no  critical  point $t \in (\dot{\mathbb{R}}^+)^n $ such that $t\cdot u \in \mathcal{N}_{r(t\cdot u)}$; (ii)  $r(t\cdot u)$
has only one  critical  point $t_u \in (\dot{\mathbb{R}}^+)^n $ such that $t_uu \in \mathcal{N}_{r(t\cdot u)}$, moreover,	$r(t\cdot u)$ attains its global minimum or maximum at $t_u$ in $(\dot{\mathbb{R}}^+)^n $; (iii) $\nabla_tr(t\cdot u)\equiv 0_n$ for all $t \in (\dot{\mathbb{R}}^+)^n $ that is $r(t\cdot u)$ is a constant in $(\dot{\mathbb{R}}^+)^n $ and attains its global minimum and maximum at any  $t \in (\mathbb{R}^+)^n $.
\end{rem}
\begin{rem}\label{R2}
	In  the  case of scalar NMM, since $tu \in \mathcal{N}^{sc}_{r(t u)}$ for any $u \in \mathcal{W}$, $t>0$,  condition {\bf (S)} can be written as follows: for every $u \in\mathcal{W}$ one of the following holds: (i) $r(t u)$ has no  critical  points in $\dot{\mathbb{R}}^+$; (ii)  $r(t u)$
has only one  critical  point $t \in \dot{\mathbb{R}}^+$; (iii) $\nabla_tr(tu)\equiv 0$ for all $t \in \dot{\mathbb{R}}^+$.
\end{rem}
\begin{rem}\label{RemScal3}
There is a special case where $r(tu)$ satisfies the following shape condition: (s00) has no critical point in $\dot{\mathbb{R}}^+$ for any $u\in \mathcal{W}$. This implies, in particular, that for each $u\in \dot{W}$ the function  $\alpha_u(t):=\Phi_\lambda (tu)$  has only one extreme value in  $\dot{\mathbb{R}}^+$. The latter condition, in turn, means that in fact we are dealing with nonparametric Nehari manifolds.  Observe, as per Poincare's program on the hierarchy of degeneracy by codimension (see \cite{arnold}), first we have to study the generic cases, that is nonparametric Nehari manifolds. However, the present work does not focus on this special case. For further insight into this problem, we refer the reader to \cite{szulkin}.
 \end{rem}

Introduce
\begin{align}
	\lambda^*_{min}&=\sup_{u \in \mathcal{W}}\,\lambda(u)\equiv \sup_{u \in \mathcal{W}}\,\inf_{t\in  (\mathbb{R}^+)^n}\, r(t\cdot u), \label{lammin}\\
	\lambda^*_{max}&=\inf_{u \in \mathcal{W}}\,\Lambda(u)\equiv\inf_{u \in \mathcal{W}}\,\sup_{t\in  (\mathbb{R}^+)^n}\, r(t\cdot u). \label{lammax}
\end{align}
Similarly, define $\lambda^{sc,*}_{min}$, $\lambda^{sc,*}_{max}$ for the scalar fibered NG-Rayleigh's quotient. 

Evidently $\lambda_{min}\leq \lambda^*_{min}$ and $\lambda^*_{max}\leq \lambda_{max}$. Consequently, in  the  case of scalar NMM, if 
 $\lambda^{sc,*}_{min} < \lambda^{sc,*}_{max} $, then for any $\lambda \in (\lambda^{sc,*}_{min}, \lambda^{sc,*}_{max})$ the Nehari manifold $\mathcal{N}^{sc}_\lambda$ is nonempty. 

\begin{lem}\label{lemthm1}
	Suppose  {\rm \bf (A)}, {\rm\bf (S)} are satisfied and $\lambda^*_{min} < \lambda^*_{max} $. Then for any $u \in \mathcal{N}_\lambda$,  $ \det	J(\nabla_u\Phi_{\lambda}(u)(u))\neq  0$ provided $\lambda \in (\lambda^*_{min}, \lambda^*_{max})$.
		\end{lem}
\begin{proof}
Let $\lambda \in (\lambda^*_{min}, \lambda^*_{max})$ and $u \in \mathcal{N}_\lambda$.
Suppose, contrary to our claim, that $ \det	J(\nabla_u\Phi_{\lambda}(u)(u))= 0$. Note that $\lambda=r(u)$, since $u \in \mathcal{N}_\lambda$. Therefore by {\rm \bf (A)}, the point $t_{u}= 1_n$ is an extremal for  the function $r(t u)$. Then  {\bf (S)} yields that at the point $t_{u}= 1_n$ the function $r(t\cdot u)$ attains its global minimum or/and maximum. Assume for instance that this is a global minimum point. Since $\lambda>\lambda^*_{min}$, by  \eqref{lammin} one has 
$$\min_{t\in  (\mathbb{R}^+)^n}r(t\cdot u)=\lambda> \inf_{ t \in (\mathbb{R}^+)^n}\, r(t \cdot u).
$$ 
Thus we get a contradiction. The same conclusion can be drawn when $t_{u}= 1_n$ is the point of global maximum of $r(t\cdot u)$.
\end{proof}

\begin{thm}\label{thm1}
	Suppose   {\rm \bf (A)}, {\rm\bf (S)} are satisfied and $\lambda^*_{min} < \lambda^*_{max} $. Assume  $\lambda \in (\lambda^*_{min}, \lambda^*_{max})$.  Then any  solution $u_\lambda$ of \eqref{N11} 
		  satisfies   equation \eqref{V1}.
\end{thm}
\begin{proof} Let $\lambda \in (\lambda^*_{min}, \lambda^*_{max})$ and $u_\lambda$ be a  solution of \eqref{N11}. Then Lemma \ref{lemthm1} yields that  $u_\lambda$ satisfies  condition \eqref{Cond} and therefore by Lemma \ref{lem1} the theorem follows.
\end{proof}

Consider the following particular case of {\bf (S)}: 

\medskip
\begin{description}{\it
		\item[(S0)]  For any $u\in \dot{W}$ one of the following holds:
			\begin{description}
			\item[(i)] $r(t \cdot u)$ has no  critical  point $t \in (\dot{\mathbb{R}}^+)^n $ such that $tu \in \mathcal{N}_{r(t\cdot u)}$;
			\item[(ii)] $\nabla_tr(t\cdot u)\equiv 0_n$ for all $t \in (\dot{\mathbb{R}}^+)^n $.
		\end{description}
		 }
		\end{description}
\medskip

\begin{rem}\label{RemScal2}
Taking into account Remark \ref{RemScal1} we see that, in  the  case of scalar NMM, condition {\rm \bf (S0)} implies 	{\rm \bf (A)}. 
\end{rem}

\begin{thm}\label{thm2}
	Suppose    {\rm \bf (A)}, {\rm\bf (S0)} hold and $-\infty\leq \lambda_{min} < \lambda^*_{max}$. Assume $\lambda \in (\lambda_{min}, \lambda^*_{max})$. Then any  solution $u_\lambda$ of \eqref{N11} 
		  satisfies   equation \eqref{V1}.
\end{thm}
\begin{proof}  As above, to prove the assertion it is sufficient to show that condition \eqref{Cond} is satisfied. Assume $ \det	J(\nabla_u\Phi_{\lambda}(u)(u))= 0$. By {\bf (A)}, $t_{u_\lambda}= 1_n$ is the extremal point for  the function $r(t\cdot u)$ and consequently $\nabla_tr(t\cdot u)|_{t=1_n}= 0_n$. Since $u_\lambda \in \mathcal{N}_\lambda\equiv \mathcal{N}_{r(u)}$, {\bf (S0)} entails that the function $r(t\cdot u_\lambda)$ identically equals to the constant $\lambda$ in $(\mathbb{R}^+)^n $ and attains its global minimum and maximum at any point $t \in (\mathbb{R}^+)^n $.  However, the assumption $\lambda<\lambda^*_{max}$  yields that $\lambda< \sup_{ t \in (\mathbb{R}^+)^n}\, r(t\cdot u_\lambda)=\max_{ t \in (\mathbb{R}^+)^n}\, r(t\cdot u_\lambda)\equiv r(u_\lambda)=\lambda$. Thus we get a contradiction.  
\end{proof}
Clearly, this proof  also contains
	\begin{lem}\label{ldet}
	Suppose  {\rm \bf (A)}, {\rm\bf (S0)} hold and $\lambda_{min} < \lambda^*_{max} $. Then for any $u \in \mathcal{N}_\lambda$,  $ \det	J(\nabla_u\Phi_{\lambda}(u)(u))\neq  0$ provided $\lambda \in (\lambda_{min}, \lambda^*_{max})$.
		\end{lem}

\begin{rem}\label{RExtr}
In the present work we do not deal with the applicability of Nehari manifolds method at its  extremal points like $\lambda_{max}^*$, $\lambda_{max}$	or $\lambda_{min}$. This is a subject of another work. 
\end{rem}

Arguing as in the proof of Proposition \ref{prop2}, it can be proved 
\begin{cor}\label{Csman}
Suppose   {\rm \bf (A)}, {\rm\bf (S)} {\rm({\bf (S0)})}  hold and $\lambda^*_{min} < \lambda^*_{max} $ $(\lambda_{min} < \lambda^*_{max}) $. Then $\mathcal{N}_\lambda$  is a $C^1$-manifold provided $\lambda \in (\lambda^*_{min}, \lambda^*_{max})$ $(\lambda \in (\lambda_{min}, \lambda^*_{max}))$.
\end{cor}
 
Let us stress that 
$$
\Lambda^{sc}(u):=\sup_{t\in  \mathbb{R}^+}\, r(tu)\leq \sup_{t\in  (\mathbb{R}^+)^n}\, r(t\cdot u)=:\Lambda(u),
$$
and therefore,
\begin{equation}\label{nerav}
	\lambda^{sc,*}_{max}:=\inf_{u \in \mathcal{W}}\,\Lambda^{sc}(u)\leq \inf_{u \in \mathcal{W}}\,\Lambda(u)=:\lambda^*_{max}. 
\end{equation}
Similarly,
\begin{equation}\label{nerav2}
\lambda^{sc}_{max}\leq \lambda_{max},~	\lambda_{min}\leq \lambda^{sc}_{min},~ \lambda^*_{min}\leq \lambda^{sc,*}_{min}.
\end{equation}

\begin{rem} \label{RemInv1}
It is obviously that the assumptions {\bf (S)},  {\bf (S0)} and the  definition of the extremal values $\lambda_{min}, \lambda_{max}, \lambda^*_{min},\lambda^*_{max}$ etc. do not depend on of the change of variables as to Remark \ref{RemInv}.
\end{rem}

\section{Extreme values of NMM   in explicit  variational forms}
 
In this section,  applying the above theory we present some examples where the extreme values of NMM can be expressed  in an explicit  variational form.

\medskip

\noindent
 \emph{ \bf Example 1. (Problem with indefinite nonlinearity)}

\medskip

Consider the following
boundary value problem with indefinite nonlinearity
\begin{equation}
\label{pb1} \left\{
\begin{array}{l}
\ -\Delta_p u=\lambda
|u|^{p-2}u+f(x)|u|^{\gamma-2}u~~\mbox{in}~~\Omega, \\ 
~~~u=0~~\mbox{on}~~\partial\Omega,
%
\end{array}
\right.
\end{equation}
where $\Omega$ is a bounded domain in $\mathbb{R}^N$ with smooth
boundary; $\lambda\in \mathbb{R}$; $\Delta_p(\cdot):=\mbox{div}(|\nabla (\cdot)|^{p-2}\nabla (\cdot))$ is the $p$-Laplacian;   $f\in L^{\infty}(\Omega)$ and we assume 
\begin{equation}
1<p<\gamma \leq p^*, ~\mbox{where }~ p^*=
\left\{
\begin{array}{ll}
\frac{pN}{N-p} & \mbox{ if }~ p<N,\\
+\infty  & \mbox{ if }~p\geq N. \label{200}
\end{array}
\right.
\end{equation}
Subsequently, $W:=W^{1,p}_0(\Omega)$ denotes the standard Sobolev space  with the norm
$$
||u||_1=(\int_{\Omega} |\nabla u|^p\,dx)^{1/p}.
$$
By a solution of \eqref{pb1} we shall  mean a weak solution $u \in W:= W^{1,p}_0(\Omega)$. In what follows, $\lambda_1:=\lambda_{1,p}$, $\phi_1:=\phi_{1,p}$ denote the first eigenpair of the operator $-\Delta_p$ in $\Omega$ with zero boundary conditions. It is known that the eigenvalue $\lambda_1$ is positive, simple and isolated, the corresponding eigenfunction  $\phi_1$ is positive and it can be normalized so that $||\phi_1||_1=1$ \cite{Anan, DrabekT, lindqvist}. In the case when $f$ may change the sign in $\Omega$, the nonlinearity  in right hand side of \eqref{pb1} is called indefinite in sign (cf. \cite{Alam, BCDN}). 

	The boundary value problems with indefinite nonlinearity have been studied in a number of papers (see e.g. \cite{Alam, BCDN, Bozhkov, DrabPoh, ilSari, ilrunst, ilEg, Ou2}). An important role in these studies plays the following extremal value
\begin{equation}
\lambda^*=\inf\{\frac{\int_\Omega |\nabla u|^p dx }{\int_\Omega
|u|^{p}dx}:~~\int_\Omega
f(x)|u|^{\gamma}dx\geq  0,~u \in W\setminus 0\}, \label{Ouyang}		
\end{equation}
which, as far as we know, was first found by Ouyang \cite{Ou2}.

Note that
$$
	\lambda^*>\lambda_1~~\mbox{if and only if}~~\int_\Omega
f(x)|\phi_1|^{\gamma}dx< 0.
$$
Let us show that \eqref{Ouyang} can be obtained by applying the method of NG-Rayleigh's quotient.

\begin{lem}
	$\lambda^*$ is the extreme value of NMM, namely it coincides with \eqref{lammax}, i.e. $\lambda^*=\lambda^*_{max}$.
\end{lem}
\begin{proof}
Let
\begin{equation}\label{pb111}
\Phi_\lambda (u) = \frac{1}{p} \int |\nabla u|^{p} dx - \lambda\frac{1}{p} \int |u|^{p} dx -
\frac{1}{\gamma} \int f|u|^{\gamma} dx,~u \in W.
\end{equation}
Evidently, $\Phi_\lambda \in  C^1(W, \mathbb{R})$, $\partial_t \tilde{\Phi}_\lambda \in C^1(\mathbb{R}^+ \times W, \mathbb{R})$. Note that $\Phi_\lambda \not\in  C^2(W\setminus 0, \mathbb{R})$ if $1<p<2$.

Consider the corresponding Nehari manifold minimization problem 
\begin{align}
		\begin{cases}\label{N1ind}
	&\frac{1}{p} \int |\nabla u|^{p} dx - \lambda\frac{1}{p} \int |u|^{q} dx -
\frac{1}{\gamma} \int f|u|^{\gamma} dx \to min\\	\\
	&  \int |\nabla u|^{p} dx - \lambda \int |u|^{p} dx -
 \int f|u|^{\gamma} dx=0,~~u\in W\setminus 0.
	\end{cases}
	\end{align}
Consider the  NG-Rayleigh's quotient corresponding to \eqref{pb1}
\begin{equation}\label{rind}
	r(u)=\frac{\int_\Omega |\nabla u|^p dx -\int_\Omega
f(x)|u|^{\gamma}dx}{\int_\Omega
|u|^{p}dx},~~u \in W\setminus 0
\end{equation}
Let $u \in W\setminus 0$. Observe, for $t>0$
\begin{equation}
	\label{rindt}
	r(t u)=\frac{\int_\Omega |\nabla u|^p dx }{\int_\Omega
	|u|^{p}dx}- t^{\gamma-p}\frac{\int_\Omega
	f(x)|u|^{\gamma}dx}{\int_\Omega
	|u|^{p}dx}.
\end{equation}
Hence, 
\begin{equation*}
\Lambda(u)=\sup_{t>0}\, r(tu)=\left\{
\begin{array}{l}
\ \frac{\int_\Omega |\nabla u|^p dx }{\int_\Omega
|u|^{p}dx}, ~~\mbox{if}~~\int_\Omega f(x)|u|^{\gamma}dx\geq 0, \\ \\
~~+\infty, ~~\mbox{if}~~\int_\Omega
f(x)|u|^{\gamma}dx< 0,
%
\end{array}
\right.
\end{equation*}
and 
\begin{align*}
	\lambda(u)=	\inf_{t>0}\, r(tu)=\begin{cases}
	&  -\infty,~~\mbox{if}~~\int_\Omega f(x)|u|^{\gamma}dx> 0, \\ \\
&\frac{\int_\Omega |\nabla u|^p dx }{\int_\Omega
|u|^{p}dx}, ~~\mbox{if}~~\int_\Omega
f(x)|u|^{\gamma}dx\leq  0.
	\end{cases}
	\end{align*}
Consequently, we obtain for extreme values \eqref{lammin}, \eqref{lammax} the following explicit variational forms
\begin{align}
	\lambda^*_{min}&=\sup\{\frac{\int_\Omega |\nabla u|^p dx }{\int_\Omega
|u|^{p}dx}:~~\int_\Omega
f(x)|u|^{\gamma}dx\leq  0,~u \in W\setminus 0\}, \nonumber\\ 
	\lambda^*_{max}&=\inf\{\frac{\int_\Omega |\nabla u|^p dx }{\int_\Omega
|u|^{p}dx}:~~\int_\Omega
f(x)|u|^{\gamma}dx\geq  0,~u \in W\setminus 0\}. \label{Ouyang2}
\end{align}
Thus we see that \eqref{Ouyang2} coinsides with \eqref{Ouyang}. 
\end{proof}

It is easily to see that $r(tu)$ has only extremal point at $t=0$ or in the case $\int_\Omega
	f(x)|u|^{\gamma}dx=0$,  $r(tu) \equiv \frac{\int_\Omega |\nabla u|^p dx }{\int_\Omega
	|u|^{p}dx}$ for all $t\geq 0$. Thus  condition {\bf (S0)} is satisfied and consequently {\bf (A)} holds (see Remark \ref{RemScal2}). 
	
Observe,  
$\lambda^*_{min}=-\infty$ if  $\int_\Omega
f(x)|u|^{\gamma}dx>0$ for all $u \in W\setminus 0$, and $\lambda^*_{min}=+\infty$ if
the set $\{x \in \Omega: ~f(x)\leq 0\}$ contains an open domain up to a subset of Lebesgue measure  zero.  Note that  when  $\int_\Omega
f(x)|u|^{\gamma}dx>0$ for all $u \in W\setminus 0$ we have $\lambda^*_{max}=\lambda_1$.
Thus, we have a strong inequality $\lambda^*_{min}<\lambda^*_{max}=\lambda_1$ only if $f>0$ a.e. in $\Omega$, and therefore only in this case we may apply Theorem \ref{thm1}.

However, for another extremal value of NMM $\lambda_{min}:=\inf_{u \in W\setminus 0}\lambda(u)$ we have $\lambda_{min}=-\infty$ if  $\int_\Omega
f(x)|u|^{\gamma}dx>0$ for some $u \in W\setminus 0$, and $\lambda_{min}= \lambda_1$, $\lambda^*_{max}=+\infty$  if  $\int_\Omega f(x)|u|^{\gamma}dx\leq 0$ for all $u \in W\setminus 0$. Thus we always have  $\lambda_{min}<\lambda^*_{max}$. Thus Theorem \ref{thm2} yields
\begin{lem}
Assume $1<p<\gamma \leq {p}^*$ and $\lambda\in (\lambda_{min},\lambda^*_{max})$. Then any solution $u_\lambda$ of \eqref{N1ind} satisfies  equation \eqref{pb1}.
\end{lem}
The proof of the existence of the solution of  Nehari manifold minimization problem \eqref{N1ind} when $\lambda\in  (\lambda_{min},\lambda^*_{max})$ can be found in \cite{Alam,Ou2} (for $\gamma<p^*$, $p=2$), in  \cite{DrabPoh, ilIzv, ilrunst, ilrunst2} ( for $1<p<+\infty$ ) and in  \cite{ilEg} (for $\gamma \leq {p}^*$ ).   
Furthermore, it can be proved that if $\lambda \in (\lambda_1, \lambda^*_{max})$ then \eqref{N1ind} has two  nonnegative solutions (see e.g.  \cite{Alam, Ou2, ilIzv}).

\medskip

\noindent
\emph{ \bf Example 2. (Problem with convex-concave nonlinearity)}

\medskip

Consider the following problem with convex-concave nonlinearity
\begin{equation}
\label{pb} \left\{
\begin{array}{l}
\
-\Delta_p u = \lambda |u|^{q-2}u+f(x)|u|^{\gamma-2}u,
~ x \in \Omega, \\ \\
\hspace*{0.2cm} u|_{\partial \Omega} = 0,
\end{array}
\right.
\end{equation}
where   $\Omega$ is a bounded domain in $\mathbb{R}^N$,
$N\geq 1$, with smooth boundary $\partial \Omega$,
$\Delta_p$ is the p-Laplacian and we assume that $1<q<p<\gamma \leq p^*$.

We always suppose that $f(x) \geq 0$ on $\Omega$ and $f \in L^d(\Omega)$, where $d>p^*/(p^*-\gamma)$ if $p<N$ and $\gamma<p^*$; $d=+\infty$ if $p<N$ and $\gamma=p^*$; $d>1$ if $p\geq N$.
By a solution of \eqref{200} we shall mean a weak solution $u \in  W^{1,p}_0(\Omega)$.

The investigation of the problems with convex-concave type nonlinearity similar to \eqref{pb} can be found in various papers
(see e.g. \cite{AmAzPer, AmBrCer, ilconcan, KentWu, LiWang}). In particular, in \cite{ilconcan} has been found using the so-called  spectral analysis by the fibering method the following extremal value of NMM
\begin{eqnarray}
\lambda^*= 
\frac{\gamma-p}{p-q} \left( \frac{p-q}{\gamma-q}
\right)^{\frac{\gamma-q}{\gamma-p}}  \inf_{u \in W\setminus 0}	 
\left( \frac{(\int|\nabla u|^{p} dx)^{\frac{\gamma-q}{\gamma-p}}}{(\int |u|^{q} dx) \,
(\int f|u|^{\gamma} dx)^{\frac{p-q}{\gamma-p}}  } \right).\label{LMaxN}
\end{eqnarray}  
Let us show that \eqref{LMaxN} can be obtained also by applying the method of NG-Rayleigh's quotient.
\begin{lem}
	$\lambda^*$ is the extreme value of NMM, namely it coincides with \eqref{lammax}, i.e. $\lambda^*=\lambda^*_{max}$ .
\end{lem}
\begin{proof} 
Let 
\begin{equation}\label{pb11}
\Phi_\lambda (u) = \frac{1}{p} \int |\nabla u|^{p} dx - \lambda\frac{1}{q} \int |u|^{q} dx -
\frac{1}{\gamma} \int|u|^{\gamma} dx.
\end{equation}
The  Nehari manifold minimization problem for \eqref{pb} is given by
\begin{align}
		\begin{cases}\label{N1cc}
	&\frac{1}{p} \int |\nabla u|^{p} dx - \lambda\frac{1}{q} \int |u|^{q} dx -
\frac{1}{\gamma} \int|u|^{\gamma} dx \to min\\	\\
	&  \int |\nabla u|^{p} dx - \lambda \int |u|^{q} dx -
 \int f|u|^{\gamma} dx=0,~~u\in \dot{W},
	\end{cases}
	\end{align}
and the corresponding NG-Rayleigh's quotient is 
\begin{equation}\label{lamb1}
	r(u)=\frac{\int |\nabla u|^{p} dx-\int |u|^{\gamma} dx}{\int f|u|^{q} dx}.
\end{equation}
Observe for $u \in W\setminus 0$,  $t>0$
\begin{equation}\label{RaylCC}
	r(t u)=\frac{t^{p-q}\int|\nabla u|^{p} dx-t^{\gamma-q}\int f|u|^{\gamma} dx}{\int |u|^{q} dx}.
\end{equation}
Compute
\begin{equation}
	\frac{\partial}{\partial t} r(t u)=\frac{(p-q)t^{p-q-1}\int|\nabla u|^{p} dx-(\gamma-q)t^{\gamma-q-1}\int f|u|^{\gamma} dx}{\int |u|^{q} dx}
\end{equation}
Hence,  $\frac{\partial}{\partial t} r(t u)=0$ if and only if 
$$
(p-q)t^{p-q-1}\int|\nabla u|^{p} dx-(\gamma-q)t^{\gamma-q-1}\int f|u|^{\gamma} dx=0.
$$
The only solution of this equation is 
\begin{equation}\label{tmaxcc}
	t_{u,max}=\left(\frac{(p-q)\int|\nabla u|^{p} dx}{(\gamma-q)\int f |u|^{\gamma} dx}\right)^\frac{1}{\gamma-p}.  
\end{equation}
The substituting  $t_{u,max}$ into $r(t u)$ yields
\begin{equation}\label{LambCc}
	\Lambda(u)=r(t_{u,max} u)=\frac{\gamma-p}{p-q} \left( \frac{p-q}{\gamma-q}
	\right)^{\frac{\gamma-q}{\gamma-p}} 
	\left( \frac{(\int|\nabla u|^{p} dx)^{\frac{\gamma-q}{\gamma-p}}}{(\int |u|^{q} dx) \,
	(\int |u|^{\gamma} dx)^{\frac{p-q}{\gamma-p}}  } \right)
\end{equation}
Thus,  indeed $\lambda^*_{max}=\inf_{u \in \dot{W}}\,\Lambda(u)$ coincides with \eqref{LMaxN}. 
\end{proof}

Obviously, conditions {\bf (A)} and  {\bf (S)} hold (see Remark \ref{RemScal1}). 
It is not hard to show using Sobolev's imbedding theorem (see \cite{ilconcan}) that $\lambda^*_{max}>0$.  Notice  
$$
\lambda(u)=\inf_{t>0}r(t u)=-\infty, ~~~\forall u \in W\setminus 0,
$$
and therefore $\lambda_{\min}=\lambda^*_{\min}=-\infty$. 
Thus  Theorem \ref{thm1} yields
\begin{lem}\label{lemCC}
Assume \eqref{200} holds. Then $\lambda^*_{max}>0$  and for $\lambda \in (-\infty, \lambda^*_{max})$ any solution $u_\lambda$ of \eqref{N1cc} satisfies equation \eqref{pb}.
\end{lem}
For the existence of the solution of \eqref{N1cc} when $\lambda\in  (-\infty, \lambda^*_{max})$   we refer the reader to  \cite{AmAzPer, AmBrCer, ilconcan, KentWu} where the existence of two distinct nonnegative solutions \eqref{pb} for $\lambda \in (0, \lambda^*_{max})$ is proven as well. 
However, the proof of these assertions can be found also below in Section 7, where we prove similar results for  general convex-concave problems which contain  \eqref{N1cc} as a particular case.

\medskip

\noindent
 \emph{ \bf Example 3. (System of equations with convex-concave nonlinearity)}

\medskip

Consider system of equations with convex-concave nonlinearity
\begin{equation}
\label{DCC}
\begin{cases}
  -\Delta_p u = \lambda |u|^{q-2} u + \alpha f(x)|u|^{\alpha-2} u |v|^{\beta}, \quad x \in \Omega, \\[0.5em]
  -\Delta_p v = \lambda |v|^{q-2} v + \beta f(x)|u|^{\alpha} |v|^{\beta-2} v, \quad x \in \Omega, \\[0.5em]
    u|_{\partial \Omega} = 0,~v|_{\partial \Omega} = 0,
\end{cases}
\end{equation}
where $\Omega \subset \mathbb{R}^N$, $N\geq 1$ is a bounded domain with $C^1$-boundary $\partial \Omega$, $\lambda, \mu \in \mathbb{R}$, $1<q< p< \alpha+\beta\leq p^*$.
We suppose $f \in L^\infty(\Omega)$ and $f \geq 0$ in $\Omega$. By a solution of \eqref{pb1} we shall mean a weak solution $(u,v) \in W:= W^{1,p}_0(\Omega) \times W^{1,q}_0(\Omega)$.

The problem has a variational form with the
Euler-Lagrange functional given by
\begin{eqnarray*}
\Phi_{\lambda}(u,v)=\frac{1}{p}\int( |\nabla u|^p +|\nabla v|^p) dx+\lambda\frac{1}{q}\int (|u|^q +\int |v|^q) dx -\int f(x)|u|^{\alpha}|v|^{\beta}dx,
\end{eqnarray*}
for $(u,v) \in W$. Evidentely, 
$\Phi_\lambda \in  C^1(\dot{W}, \mathbb{R})$, $\nabla_t \tilde{\Phi}_\lambda \in C^1((\dot{\mathbb{R}}^+)^n \times \dot{W}, \mathbb{R}^n)$. However, $\Phi_\lambda \not\in  C^2(\dot{W}, \mathbb{R})$ if $1<p<2$ or/and $1<q<2$.

The corresponding NG-Rayleigh's quotient is defined as follows
\begin{equation*}
r(u,v)=\frac{\int |\nabla u|^p dx +  \int |\nabla v|^p dx -(\alpha+\beta) \int f(x)|u|^{\alpha}|v|^{\beta}dx}{ \int |u|^q dx +  \int |v|^q dx}.	
\end{equation*}
For \eqref{DCC} we may apply  two methods:  the vector NMM \eqref{N11} or the scalar NMM \eqref{N11KW}. Let us consider both of them. 

\medskip
\par\noindent 
{\it Scalar Nehari manifold method}.
\medskip

The scalar Nehari manifold minimization problem  corresponding to \eqref{DCC} is defined as follows
\begin{align}\label{N1ccSys}
		\begin{cases}
			&\Phi_{\lambda}(u,v) \to min\\	
	 &\int( |\nabla u|^p +|\nabla v|^p) dx+\lambda\int (|u|^q +\int |v|^q) dx -\\
	&~~~~~~~~~~~~~~~~~~~~~~~~~~~~~(\alpha+\beta)\int f(x)|u|^{\alpha}|v|^{\beta}dx=0,\\
	 &~(u,v)\in \dot{W}.
	\end{cases}
		\end{align}
Consider
$$
r(t(u,v))=\frac{t^{p-q}(\int( |\nabla u|^p +|\nabla v|^p) dx-(\alpha+\beta)t^{\alpha+\beta-q}\int f(x)|u|^{\alpha}|v|^{\beta}dx}{\int (|u|^q +\int |v|^q) dx},
$$
where $t>0,~(u,v) \in W \setminus 0_2$. It is easily seen that in this case we may apply the same analysis as it has been done above for \eqref{N1cc}, \eqref{RaylCC}. In this way, we  introduce 
\begin{equation*}
	\Lambda^{sc}(u,v)=C_{p,q,\alpha,\beta}
	\left( \frac{(\int( |\nabla u|^p +|\nabla v|^p) dx)^{\frac{\alpha+\beta-q}{\alpha+\beta-p}}}{(\int (|u|^q +\int |v|^q) dx) \,
	((\alpha+\beta)\int f(x)|u|^{\alpha}|v|^{\beta}dx)^{\frac{p-q}{\alpha+\beta-p}}  } \right)
\end{equation*}
where $C_{p,q,\alpha,\beta}$ is a constant which does not depend on $(u,v) \in \dot{W}$. Thus for the extreme value \eqref{lammax} express as the following explicit variational form 
$$
\lambda^{sc,*}_{max}=\inf_{(u,v) \in W\setminus 0_2}\,\Lambda^{sc}(u,v).
$$
As above, it is not hard to show using Sobolev's imbedding theorem that $\lambda^{sc,*}_{max}>0$. Thus Theorem \ref{thm1} yields
\begin{lem}\label{lemCCSys}
Assume $1<q< p< \alpha+\beta\leq p^*$. Then $\lambda^{sc,*}_{max}>0$  and for $\lambda \in (-\infty, \lambda^{sc,*}_{max})$ any solution $u_\lambda$ of \eqref{N1ccSys} satisfies equation \eqref{DCC}.
\end{lem}
The existence of the solution for $\lambda\in  (-\infty, \lambda^{sc,*}_{max})$ (and even multiple solutions for $\lambda\in  (0, \lambda^{sc,*}_{max})$)  of \eqref{N1ccSys}  can be obtained using Theorems 3.1 and 3.2 from \cite{KentWu}. Note that in  \cite{KentWu} the existence of solutions of \eqref{DCC} is proven only locally by $\lambda$, namely, for $\lambda \in (0,\delta)$ with some sufficiently small $\delta$.

\medskip
\par\noindent 
{\it Vector Nehari manifold method}.
\medskip

Now consider  the corresponding vector Nehari manifold minimization problem
\begin{align}\label{NCS}
		\begin{cases}
	&~\Phi_{\lambda}(u,v) \to min\\	
	& \int|\nabla u|^p dx-\lambda \int|u|^q dx-\alpha \int f(x)|u|^{\alpha}|v|^{\beta}dx=0,\\~~
 &\int|\nabla v|^p dx-\lambda \int|v|^q dx-\beta \int f(x)|u|^{\alpha}|v|^{\beta}dx=0,\\
	&~(u,v) \in \dot{W}.
	\end{cases}
	\end{align}
Let us find the extremal value \eqref{lammax} corresponding to   \eqref{NCS}. Taking into account Remark \ref{RemInv1} we have
$$
\Lambda(u,v)=\sup_{(t,s)>0}r(tu,tsv)=\frac{t^p\int( |\nabla u|^p +s^p|\nabla v|^p) dx-t^{\alpha+\beta}s^\beta\int f(x)|u|^{\alpha}|v|^{\beta}dx}{t^q\int (|u|^q +s^q\int |v|^q) dx}.
$$
Similar to \eqref{LambCc} we see that for every $s>0$ there holds
$$
\sup_{t>0}r(tu,tsv)=C_{p,q,\alpha,\beta}
	\left( \frac{(\int( |\nabla u|^p +s^p|\nabla v|^p) dx)^{\frac{\alpha+\beta-q}{\gamma-p}}}{(\int (|u|^q +s^q\int |v|^q) dx) \,
	(s^\beta(\alpha+\beta)\int f(x)|u|^{\alpha}|v|^{\beta}dx)^{\frac{p-q}{\alpha+\beta-p}}  } \right).
$$	
Using this fact it is not hard to show that
$$
\lambda^*_{v, max}=\inf_{u \in W\setminus 0}\,\Lambda(u,v)=\inf_{u \in W\setminus 0}\,\Lambda(u,v)= \lambda^{sc,*}_{max}.
$$
Thus,  the extremal values for the vector \eqref{NCS} and the scalar   \eqref{N1ccSys} NMM are the same. The investigate of the existence of solution of \eqref{NCS} is left to the reader.

\medskip

\noindent
 \emph{ \bf Example 4. (System of equations with indefinite nonlinearity)}

Consider system of equations with indefinite nonlinearity
\begin{equation}
\label{D}
\begin{cases}
  -\Delta_p u = \lambda |u|^{p-2} u + \alpha f(x)|u|^{\alpha-2} u |v|^{\beta}, \quad x \in \Omega, \\[0.5em]
  -\Delta_q v = \lambda |v|^{q-2} v + \beta f(x)|u|^{\alpha} |v|^{\beta-2} v, \quad x \in \Omega, \\[0.5em]
    u|_{\partial \Omega} = 0,~v|_{\partial \Omega} = 0,
\end{cases}
\end{equation}
where $\Omega \subset \mathbb{R}^N$, $N\geq 1$ is a bounded domain with $C^1$-boundary $\partial \Omega$, $\lambda, \mu \in \mathbb{R}$, $1< p< +\infty$, $1< q < +\infty$ and
\begin{equation}
\label{Sob}
\alpha,~ \beta > 0, \qquad \frac{\alpha}{p} +\frac{\beta}{q} > 1, \qquad \frac{\alpha}{p^*} +\frac{\beta}{q^*} < 1.	
\end{equation}
Here, as above, $p^*$ and $q^*$ are the standard critical Sobolev exponents.
We suppose $f \in L^\infty(\Omega)$ and that the function $f$ may change the sign on $\Omega$, i.e. the problem \eqref{D} has indefinite nonlinearity.  By a solution of \eqref{pb1} we shall mean a weak solution $(u,v) \in W:= W^{1,p}_0(\Omega) \times W^{1,q}_0(\Omega)$.

This system of equations has been studied, for instance, in \cite{ilBobkov, ilBob, Bozhkov}. In particular, in \cite{ilBobkov} for  \eqref{D} an extreme value of Nehari manifold method has been introduced. Here we  obtain this value using the method of NG-Rayleigh's quotient.

Consider
\begin{eqnarray}\label{PhiSys}
\Phi_{\lambda}(u,v)=\frac{1}{p}P_{\lambda}(u)+\frac{1}{q}Q_{\lambda}(v) -F(u,v), \quad (u,v) \in W,
\end{eqnarray}
where 
\begin{align*}
P_{\lambda}(u) &= \int |\nabla u|^p dx-\lambda \int |u|^p dx,~~
Q_{\lambda}(v) = \int |\nabla v|^q dx- \lambda \int |v|^q dx, \\
F(u,v) &= \int f(x)|u|^{\alpha}|v|^{\beta}dx.
\end{align*}
Then the corresponding Nehari manifold is defined as follows
\begin{eqnarray*}
\mathcal{N}_{\lambda}:=\{(u,v) \in \dot{W}:~~P_{\lambda}(u)- \alpha F(u,v)=0,~~
 Q_{\lambda}(v)- \beta F(u,v)=0\},
\end{eqnarray*}
whereas the Nehari manifold minimization problem is
\begin{align}\label{NSys}
		\begin{cases}
	&~\frac{1}{p}P_{\lambda}(u)+\frac{1}{q}Q_{\lambda}(v) -F(u,v) \to min\\	
	& \begin{array} {lcl} P_{\lambda}(u)- \alpha F(u,v)&=&0, 
	\\Q_{\lambda}(u)- \beta F(u,v)&=&0, \end{array}\\
	&~(u,v) \in \dot{W}.
	\end{cases}
	\end{align}
	Let us prove
\begin{lem}
 The extreme value \eqref{lammax} of \eqref{NSys}  is expressed  by the following explicit variational form
$$
\lambda^*_{\max}=\inf \left\{ \max \left\{ \frac{\int |\nabla u|^p dx}{\int |u|^p dx},  \frac{\int |\nabla v|^q dx}{\int |v|^q dx} \right\}: F(u,v) \geq 0, (u,v) \in \dot{W}\right\}.
$$	
\end{lem}
	\begin{proof}
		
Consider NG-Rayleigh's quotient 
\begin{equation*}
r(tu,sv)=\frac{t^p \int |\nabla u|^p dx + s^q \int |\nabla v|^q dx - t^\alpha s^\beta(\alpha+\beta) F(u,v)}{t^p \int |u|^p dx + s^q \int |v|^q dx}.	
\end{equation*}
We claim that
\begin{equation}\label{LS}
\Lambda(u,v)=\sup_{(t,s)>0}\, r(tu,sv)=\left\{
\begin{array}{l}
\ \max \left\{ \frac{\int |\nabla u|^p dx}{\int |u|^p dx},  \frac{\int |\nabla v|^q dx}{\int |v|^q dx} \right\}, \mbox{if}~F(u,v)\geq 0, \\ \\
~~+\infty, ~~\mbox{if}~~F(u,v)< 0,
%
\end{array}
\right.
\end{equation}
and 
\begin{equation}\label{lS}
	\lambda(u,v)=\inf_{(t,s)>0}\, r(tu,sv)=\begin{cases}
	&  -\infty,~~\mbox{if}~~F(u,v)> 0, \\ \\
& \min \left\{ \frac{\int |\nabla u|^p dx}{\int |u|^p dx},  \frac{\int |\nabla v|^q dx}{\int |v|^q dx} \right\}, ~\mbox{if}~~F(u,v)\leq  0.
	\end{cases}
	\end{equation}
Let us show, as an example, \eqref{LS}.  Assume $F(u,v)< 0$. Then setting $t = \sigma^q$, $s = \sigma^p$ gives
$$
\frac{\int |\nabla u|^p dx + \int |\nabla v|^q dx - \sigma^{ p q (\frac{\alpha}{p} + \frac{\beta}{q} - 1)} F(u,v)}{\int |u|^p dx + \int |v|^q dx} \to +\infty,
$$
for $\sigma \to +\infty$, since $\frac{\alpha}{p} + \frac{\beta}{q} > 1$. Consider now the case $F(u,v)\geq 0$. Without loss of generality, we can suppose that  $\frac{\int |\nabla u|^p dx}{\int |u|^p dx}\geq  \frac{\int |\nabla v|^q dx}{\int |v|^q dx}$.  This implies that
$$
\frac{\int |\nabla u|^p dx + \tau \int |\nabla v|^q dx}{\int |u|^p dx + \tau \int |v|^q dx} \leq  \frac{\int |\nabla u|^p dx}{\int |u|^p dx} 
$$
for any $\tau\geq 0$. Hence, since  $F(u,v)\geq 0$,
\begin{align*}
	r(tu,sv)=\frac{ \int |\nabla u|^p dx + s^q t^{-p}\int |\nabla v|^q dx - t^{\alpha-p} s^\beta(\alpha+\beta) F(u,v)}{ \int |u|^p dx + s^q t^{-p} \int |v|^q dx} \leq\frac{\int |\nabla u|^p dx}{\int |u|^p dx}.	
\end{align*}
for any $s\geq 0$ and $t>0$. Since for $s=0$ this inequity becomes equality, we get \eqref{LS}. Furthermore, we see that the supremum and infimum 
in \eqref{LS}, \eqref{lS} are attained on the line $s=0$ or $t=0$.

\end{proof}

Let us verify the assumption of Theorem  \ref{thm2}.
Find the corresponding Jacobian matrix
\begin{align*}
	J(\nabla_{(u,v)}&\Phi_{\lambda}(u,v)(u,v)) = \\ \\
	&\left( \begin{array}{cc}
	(p-1) P_{\lambda}(u)- \alpha (\alpha -1) F(u,v) & -\alpha \beta F(u,v)  \\ \\
	-\alpha \beta F(u,v) & (q-1) Q_{\lambda}(v)- \beta (\beta -1) F(u,v) 
	\end{array} \right).
\end{align*}
Then for  $(u,v) \in \mathcal{N}_{\lambda}$ we have 
$$
J(\nabla_{(u,v)}\Phi_{\lambda}(u,v)(u,v))  = 
\left( \begin{array}{cc}
\alpha (p - \alpha) F(u,v) & -\alpha \beta  F(u,v)  \\ \\
-\alpha \beta F(u,v) & \beta (q-\beta)F(u,v)
\end{array} \right)
$$
and consequently
$$
\mbox{det}\,J_t(\nabla_{(u,v)}\Phi_{\lambda}(u,v)(u,v)) = \alpha \beta (p q - p \beta - q \alpha) F^2(u,v).
$$
Thus, since $\frac{\alpha}{p} + \frac{\beta}{q} > 1$, if $(tu,sv) \in \mathcal{N}_{r(tu,sv)}$ and $\det	J_{(t,s)}(\nabla_{(u,v)}\Phi_{\lambda}(tu,sv)(tu,sv))= 0$ for $t>0,s>0$, then 
$F(tu,sv)=0$ and hence $J_{(t,s)}(\nabla_{(u,v)}\Phi_{\lambda}(tu,sv)(tu,sv))1_2=0_2$. Thus by (ii), Proposition \ref{grad1} it follows that $r(tu,sv)$ satisfies condition {\bf (A)}.

Observe, for $(u,v) \in W\setminus 0_2$, $t>0$, $s>0$, we have
\begin{align*}
	&\partial_t r(tu,sv)=\frac{1}{t(t^p \int |u|^p dx + s^q \int |v|^q dx)}(p  P_{\lambda}(tu)- \alpha (\alpha+\beta)  F(tu,sv)),\\ \\
	&\partial_s r(tu,sv)=\frac{1}{s(t^p \int |u|^p dx + s^q \int |v|^q dx)}(q Q_{\lambda}(sv)- \beta (\alpha+\beta)  F(tu,sv)).
\end{align*}
Thus, if we assume that $\partial_t r(tu,sv)=0$, $\partial_s r(tu,sv)=0$ and $(tu,sv) \in \mathcal{N}_{r(tu,sv)}$, then $P_{\lambda}(tu)=0$, $Q_{\lambda}(sv)=0$ and $F(tu,sv)=0$. Hence, since $t>0$, $s>0$, we have $P_{\lambda}(u)=0$, $Q_{\lambda}(v)=0$ and $F(u,v)=0$. This implies that $\partial_t r(tu,sv)\equiv 0$, $\partial_s r(tu,sv)\equiv 0$ for all $t>0$, $s>0$. Thus condition  {\bf (S0)} is satisfied.

As in the scalar case \eqref{pb1}, it is readily seen that the extremal value  $\lambda^*_{\min}=\sup_{(u,v) \in \dot{W} }\lambda(u,v)$ is useless. 

Introduce, $\lambda_{1}^{l}:=\min \{ \lambda_{1,p},\lambda_{1,q} \}$, $\lambda_1^u:=\max \{ \lambda_{1,p},\lambda_{1,q} \}$. Clearly,
\begin{align*}
	&\lambda_1^l=\inf_{(u,v) \in \dot{W}} \min \{ \frac{\int |\nabla u|^p dx}{\int |u|^p dx},  \frac{\int |\nabla v|^q dx}{\int |v|^q dx} \},\\
	&\lambda_1^u=\inf_{(u,v) \in \dot{W}} \max \{ \frac{\int |\nabla u|^p dx}{\int |u|^p dx},  \frac{\int |\nabla v|^q dx}{\int |v|^q dx} \}.
\end{align*}
Hence, $\lambda^*_{max}\geq \lambda_1^u\geq \lambda_1^l$  and
\begin{align}
	&\lambda^*_{max}>\lambda_1^u~~\mbox{iff}~~F(\phi_{1,p},\phi_{1,q})< 0.
\end{align}
Note 
\begin{align*}\label{IlC}
	\lambda_{min}= 
	\begin{cases}
		-\infty, ~~\mbox{if}~~\exists \,(u,v) \in \dot{W}~~s.t.~~ F(u,v)>0,  \\\lambda^{l}_{1},~~ ~~\mbox{if}~~\forall (u,v) \in \dot{W},~~F(u,v)\leq 0. 
	\end{cases}
\end{align*}
Thus, since $\lambda^*_{max}=+\infty$ if  $F(u,v)\leq  0$, $\forall (u,v) \in \dot{W}$, we always have  $\lambda_{min}<\lambda^*_{max}$. 

Thus, all the assumptions of Theorem  \ref{thm2} are satisfied and therefore we have
\begin{lem}
Assume\eqref{Sob} holds	and $\lambda\in (\lambda_{min},\lambda^*_{max})$. Then any solution $(u_\lambda, v_\lambda)$ of \eqref{NSys} satisfies  system \eqref{D}.
\end{lem}
The existence of the solution of \eqref{NSys} for $\lambda\in  (\lambda_{min},\lambda_1^l)\cup (\lambda_1^u,\lambda^*_{max})$  follows from  \cite{ilBobkov, ilBob}. Note that if $\lambda_{1,p}=\lambda_{1,q}$ (for instance when  $p=q$ ), then we have  $\lambda_1^l=\lambda_1^u$. However, we  conjecture that solutions of \eqref{D} can not be obtained by the Nehari manifold method if $\lambda \in [\lambda_{1}^{l},\lambda_1^u]$. Let us give some comments for this.

Observe that the system of equations
\begin{equation}
\label{sys:neh}
\begin{cases}
\frac{\partial}{\partial t}\Phi_\lambda(t u, s v)\equiv \, t^{p-1} P_{\lambda}(u)- \alpha t^{\alpha -1} s^{\beta} F(u,v)=0, \\[0.4em]
\frac{\partial}{\partial s}\Phi_\lambda(t u, s v)\equiv s^{q-1}Q_{\lambda}(v)- \beta t^{\alpha} s^{\beta-1}F(u,v)=0.
\end{cases}
\end{equation}
has a solution $(t,s) \in \dot{\mathbb{R}}^+\times \dot{\mathbb{R}}^+$ only if $(u,v)$ belongs to one of the following sets
\begin{align*}
\mathcal{A}:=\{(u,v) \in \dot{W}:~ P_{\lambda}(u)>0, \quad Q_{\lambda}(v)>0, \quad F(u,v)>0\}, \\
\mathcal{B}:=\{(u,v) \in \dot{W}:~ P_{\lambda}(u)<0, \quad Q_{\lambda}(v)<0, \quad F(u,v)<0\},
\end{align*}
Furthermore, the solution $t=t(u,v)$, $s=s(u,v)$ of \eqref{sys:neh}  is unique and given by
\begin{align}
\label{fiber_t}
t^{p q d} = \frac{\alpha^{\beta - q}}{\beta^{\beta}} \,
\frac{|P_{\lambda}(u)|^{q-\beta} |Q_{\lambda}(v)|^\beta}{|F(u, v)|^{q}},\\
\label{fiber_s}
s^{p q d} = \frac{\beta^{\alpha - p}}{\alpha^{\alpha}} \,
\frac{|P_{\lambda}(u)|^{\alpha} |Q_{\lambda}(v)|^{p - \alpha}}{|F(u, v)|^{p}},
\end{align}
where
$$
d := \frac{\alpha}{p} + \frac{\beta}{q} - 1>0.
$$
Substituting these roots into $\Phi_\lambda(t u, s v)$ we obtain the function 
\begin{equation}
\label{fiber_J}
\mathcal{J}_{\lambda}(u, v) := \Phi_\lambda(t(u,v) u, s(u,v) v) = C \frac{|P_{\lambda}(u)|^{\frac{\alpha}{p d}} |Q_{\lambda}(v)|^{\frac{\beta}{q d}}}{|F(u, v)|^{\frac{1}{d}}}\mbox{sign}(F(u, v)),
\end{equation}
where
$$
C = \left(\frac{1}{\alpha^{\alpha q} \beta^{\beta p}}\right)^{\frac{1}{p q d}}d.
$$
Note also that by the definition of $\lambda^*_{max}$ we have: for $(u,v)\in W$ 
\begin{align*}\label{IlCa}
	\mbox{if}~~F(u,v)\geq 0,~~\mbox{then}~~  
	\begin{cases}
		\int |u|^p dx \leq \frac{1}{\lambda^*_{max}} \int|\nabla u|^p dx,\\ ~~~~~~~\mbox{or/and} \\\int |v|^q dx \leq \frac{1}{\lambda^*_{max}} \int|\nabla v|^q dx,
	\end{cases}
\end{align*}
and 
\begin{align*}
	\mbox{if}~~ 
	\begin{cases}
		\int |u|^p dx > \frac{1}{\lambda^*_{max}} \int|\nabla u|^p dx,\\~~~~~~~~~~ \mbox{and} \\\int |v|^q dx > \frac{1}{\lambda^*_{max}} \int|\nabla v|^q dx,
	\end{cases}~~~\mbox{then}~~F(u,v)< 0.
\end{align*}
Let us remark that $\mathcal{A} \neq \emptyset$ only if the following is satisfied
\begin{equation}\label{LN}
	\{x \in \Omega: ~f(x)> 0\} \neq \emptyset ~~\mbox{up to a subset of Lebesgue measure  zero}
\end{equation}
Let $\lambda \in [\lambda_{1}^{l},\lambda_1^u]$. Assume, for instance, that $\lambda_{1,q}=\lambda_{1}^{l}=\min \{ \lambda_{1,p},\lambda_{1,q} \}$ and  $\lambda_{1,p}=\lambda_1^u=\max \{ \lambda_{1,p},\lambda_{1,q} \}$. Then $P_\lambda(u)>0$ for all $u \in  \dot{W}^{1,p}_0$ and therefore $\mathcal{N}_\lambda \subset  \mathcal{A}$. Furthermore, by \eqref{fiber_J}, $\Phi_\lambda(u)\geq 0$ for any $(u,v) \in \mathcal{N}_\lambda$. But then  
\begin{align}\label{Contr}
	\inf\{\Phi_\lambda(u,v): (u,v) \in \mathcal{N}_\lambda\}=0. 
\end{align}
Indeed, since $\lambda> \lambda_{1,q}$, it is not hard to find $\phi_0 \in  \dot{W}^{1,q}_0$ and a sequence $(v_m)  \subset \dot{W}^{1,q}_0$ such that $v_m \to \phi_0$ in $W^{1,q}_0$ as $m \to +\infty$,
$Q_\lambda(\phi_0)=0$, $Q_\lambda(v_m)> 0$, $m=1,2,...,$ and $Q_\lambda(v_m)\to 0$ as $m\to +\infty$. 
Hence, under assumption \eqref{LN}  we can find  $\psi_0 \in \dot{W}^{1,p}_0$ such that $F(\psi_0, \phi_0)>0$. Then $P_\lambda(\psi_0)>0$ and $F(\psi_0, v_m)\geq c_0>0$ for sufficient large $m$, since $v_m \to \phi_0$ in $W^{1,q}_0$. Setting now $u_m=\psi_0$ for $m=1,2,...$ we get by \eqref{fiber_J} that $ \Phi_\lambda(u_m,v_m)\to 0$ as   $m \to +\infty$ and thus \eqref{Contr} follows. Furthermore, by this construction we got a minimizing sequence with unsatisfactory properties. Indeed, since $\frac{\alpha}{p} +\frac{\beta}{q} > 1$, one of the following must hold
$p<\alpha$ or/and $q<\beta$. Hence \eqref{fiber_t}, \eqref{fiber_s} imply that if  $p<\alpha$, then $t_m:=t(u_m,v_m) \to 0$,  $s_m:=s(u_m,v_m) \to +\infty$ and if  $p\geq \alpha$, then $t_m \to 0$,  $s_m \to +0$. Thus in the first case we have  unbounded above minimizing sequence  $(t_m u_m, s_m v_m)$ whereas in the second case this sequence converges to zero.

Here we do not consider the application of the scalar NMM to \eqref{D}. One of the reasons is that we could not find a suitable formula for determining or evaluation the  extremal value like $\lambda^{sc,*}_{max}$. Another difficulty lies in  finding of solutions of the corresponding scalar minimization problem \eqref{N11KW}. However, by \eqref{nerav}, \eqref{nerav2} we have $(\lambda^{sc}_{min},\lambda^{sc,*}_{max}) \subseteq (\lambda_{min},\lambda^*_{max})$. Thus we should not expect that the scalar NMM will provide the better results   than by the vector NMM \eqref{NSys}.

\section{Multiplicity result}

Nehari manifold methods  is often used to prove the existence of multiple solutions, see e.g. \cite{Bozhkov, DrabPoh, ilIzv, Ou2}. In this Section, we show how to obtain such type of results using NG-Rayleigh's quotient. 

We will study  \eqref{V1} using the scalar NMM. First we prove some  general result and then we give an example to illustrate it. In what follows, we always assume that  $D_u G(u)(u) \neq 0$ for all $u \in W\setminus 0_n$ so that $\mathcal{W}=W\setminus 0_n$. We will suppose  that NG-Rayleigh's quotient $r(u)$ satisfies the following conditions: for $u\in W\setminus 0_n$

\medskip
\par
{\bf (a)}\,  
	\textit{$\tilde{r}^{sc}(t,u):=r(tu)$ attains its global maximum at a unique point $t_{u,max}>0$:\\ $r(t_{u,max}u)=\max_{t>0} r(tu)$ so that $\partial r(t_{u,max}u)/\partial t  =0$ and}
	$$
	\partial r(tu)/\partial t  >0~~\mbox{for}~~ 0<t<t_{u,max},~~\partial r(tu)/\partial t  <0~ ~\mbox{for}~ ~t>t_{u,max}.
	$$
	
 Note that if assumption {\bf (a)} holds for every $u\in W\setminus 0_n$, then {\bf (A)} and {\bf (S)} follow (see Remark \ref{R2}).
Furthermore, {\bf (a)} implies that for every $u\in W\setminus 0_n$, there exist the limits
\begin{itemize}
	\item $r(tu) \to \tilde{r}^{sc}(0,u)$ as $t\to 0$, where $-\infty\leq \tilde{r}^{sc}(0,u)<+\infty$
	\item $r(tu) \to \tilde{r}^{sc}(\infty,u)$ as $t\to +\infty$, where $-\infty\leq \tilde{r}^{sc}(\infty,u)<+\infty$.
\end{itemize}
Consider $\lambda^*_{max}=
\inf_{u \in W\setminus 0_n}\sup_{t>0} r(tu)=\inf_{u \in W\setminus 0_n}r(t_{u,max}u)$ and 
$$
\lambda^\partial_{min}=
\sup_{u \in W\setminus 0_n}\max\{\tilde{r}^{sc}(0,u),  \tilde{r}^{sc}(\infty,u)\}. 
$$
Note that {\bf (a)} implies
\begin{equation}\label{partLam}
	\lambda^*_{min}= \sup_{u \in W\setminus 0_n}\,\inf_{t>0}\, r(tu)\leq \lambda^\partial_{min}.
\end{equation}
Let us introduce the following sets
\begin{align}
&\mathcal{N}_\lambda^1:=\{u \in W\setminus 0_n:~~r(u)=\lambda,~~\partial r(tu)/\partial t|_{t=1}<0\},\label{Neh-}\\
	&\mathcal{N}_\lambda^2:=\{u \in W\setminus 0_n:~~r(u)=\lambda,~~\partial r(tu)/\partial t|_{t=1}>0\}.\label{Neh+}
\end{align}
Obviously, $\mathcal{N}_\lambda^1\cap \mathcal{N}_\lambda^2=\emptyset$ for any $\lambda< \lambda^*_{max}$, and $\mathcal{N}_\lambda^1\neq \emptyset$, $ \mathcal{N}_\lambda^2 \neq \emptyset$  for any $\lambda \in (\lambda^\partial_{min}, \lambda^*_{max})$ provided $\lambda^\partial_{min}<\lambda^*_{max}$. Moreover, for $\lambda \in \mathbb{R}$, $\mathcal{N}_\lambda^1\cup\mathcal{N}_\lambda^2=\mathcal{N}^{sc}_\lambda$ if $\mathcal{N}^{sc}_\lambda\neq \emptyset$. 
Thus, when $\mathcal{N}_\lambda^2\neq \emptyset$ and $ \mathcal{N}_\lambda^1 \neq \emptyset$  one  may split minimization problem \eqref{N12}  into the two following
\begin{align}
&\hat{\Phi}_\lambda^1:=  \min	\{\Phi_\lambda(u):~~u \in \mathcal{N}_\lambda^1\}, \label{N-}\\
&\hat{\Phi}_\lambda^2:=  \min	\{\Phi_\lambda(u):~~u \in \mathcal{N}_\lambda^2\}.\label{N+}
\end{align}
Suppose in addition to the assumptions already made: 
\medskip
\par\noindent 
{\bf (b)}\, {\it For any sequence $(v_m) \subset  W\setminus 0_n$ such that $||v_m||=1$, $m=1,2,...$ and which is weakly separated from $0_n \in W$,  the set of functions $(\psi_m(t):=r(tv_m))_{m=1}^\infty$  is bounded in $C^1[\sigma,T]$ for every $\sigma,T$ such that $T>\sigma>0$.  }
\medskip

\par\noindent 
{\bf (c)}\, {\it For any sequence $(v_m) \subset  W\setminus 0_n$, $||v_m||=1$, $m=1,2,...$ there exists $\delta>0$ such that $t_{v_m,max}>\delta$ for all $m=1,2,...$.  }
\medskip

\par\noindent 
{\bf (d)}\, {\it If $(s_m v_m) \subset \mathcal{N}^{sc}_\lambda$, $\lambda \in \mathbb{R}$ such that $||v_m||_W=1$,  $v_m \rightharpoondown 0_n$ weakly in $W$ and $\inf_{m}s_m>\delta_0>0$. Then  $s_m \to +\infty$.
 }
\medskip

\begin{thm}\label{thm3}
Suppose $W$ is a reflexive Banach space, $\Phi_\lambda \in  C^1(\dot{W}, \mathbb{R})$, $\nabla_t \tilde{\Phi}_\lambda \in C^1((\dot{\mathbb{R}}^+)^n \times \dot{W}, \mathbb{R}^n)$, $\Phi_\lambda(0_n)=0$, $D_u G(u)(u) > 0$, $\forall u\in W\setminus 0_n$, {\bf (a)}- {\bf (d)} hold and the following conditions are fulfilled:
\begin{enumerate}
\item  $ \Phi_\lambda(u) \to +\infty$ as $||u|| \to \infty$, $u\in \mathcal{N}^{sc}_\lambda$.
	\item $ \Phi_\lambda(u)$ and $r(u)$ are (sequentially) weakly lower semi-continuous functionals on $W$.
		\end{enumerate}
Assume  $\lambda^\partial_{min}<\lambda^*_{max}$.    Then for every $\lambda\in (\lambda^\partial_{min},\lambda^*_{max})$ equation \eqref{V1} possess two distinct solutions $u^1_\lambda, u^2_\lambda \in W\setminus 0_n$ such that  $u^1_\lambda$ is a minimizer  of \eqref{N-}  and  $u^2_\lambda$ is a minimizer of \eqref{N+}. Furthermore, $\partial^2 \Phi_\lambda(tu^1_\lambda)\partial t^2|_{t=1}<0$, $\partial^2 \Phi_\lambda(tu^2_\lambda)\partial t^2|_{t=1}>0$, $\Phi_\lambda(u^2_\lambda)\equiv \hat{\Phi}_\lambda^2<\Phi_\lambda(0)$ and $u^2_\lambda$ is the ground state of  \eqref{V1}.
\end{thm}
\begin{proof}  First let us show $\hat{\Phi}_\lambda^2< 0$ for $\lambda\in (\lambda^\partial_{min},\lambda^*_{max})$. Choose $u \in \mathcal{N}_\lambda^2$. Then since $r(u)=\lambda$, condition {\bf (a)} implies that $r(tu)<\lambda$ for all $t \in (0,1)$. Consequently by Corollary \ref{LP0},  $\partial \Phi_\lambda(tu)/\partial t< 0$ for all $t \in (0,1)$ and therefore $0=\Phi_\lambda(0u)>\Phi_\lambda(u)\geq \hat{\Phi}_\lambda^2$.

Let  $(u_m)$ be a minimizing sequence $(u_m)$ of \eqref{N-} or \eqref{N+}, where $\lambda\in (\lambda^\partial_{min},\lambda^*_{max})$. 
From  (1) it follows that $(u_m)$  is bounded in $W$. Since $W$ is reflexive Banah space, by the Eberlein-Smulian theorem we may assume that $u_m \rightharpoondown u_0$ weakly in $W$ as $m \to \infty$ for some $u_0 \in W$. Write $u_m =s_m v_m$ where $s_m=||u_m||$, $m=1,2,...,$. Then $(s_m)$ is bounded above and we may assume $s_m \to s_0$,  $v_m \rightharpoondown v_0$ weakly in $W$ as $m \to \infty$ for some $s_0 \geq 0$ and $v_0 \in W$.

Let us show that $s_0\neq 0$ and $v_0 \neq 0_n$. Consider first \eqref{N-}. In view of {\bf (a)}, we have  $s_m >t_{v_m,max}$. Therefore by {\bf (c)}, $\inf_m s_m >\delta_0>0$ for  some $\delta_0>0$ and thus $s_0\neq 0$. Suppose, contrary to our claim, that 
$v_m \rightharpoondown 0_n$ weakly in $W$.  Then {\bf (d)} implies $s_m \to +\infty$. Hence we get a contradiction and thus $v_0 \neq 0_n$.

 Let now $(u_m=s_mv_m)$ be a minimizing sequence of \eqref{N+}.  Assume $s_m \to 0$. Then $||u_m|| \to 0$ and $\Phi_\lambda(u_m) \to 0$.
However,   $\hat{\Phi}_\lambda^2< 0$ and we obtain a contradiction. Thus $s_0\neq 0$. Suppose, contrary to our claim, that  $v_m \rightharpoondown 0_n$ weakly in $W$.
 Then since $\inf_m s_m >\delta_0'>0$ for  some $\delta_0'>0$, we get by {\bf (d)}  a contradiction.
Thus in this case we have also $v_0 \neq 0_n$. 

Thus the set $B:=(v_m)$ is bounded above and weakly separated from $0_n \in W$. Then  {\bf (b)} implies that the set of functions  $(\psi_m(t):=r(t v_m))_{m=1}^\infty$  is bounded in $C^1[\sigma,T]$ for any $T>\sigma>0$. Consequently by the Arzel\'a-Ascoli compactness criterion  we can assume that
\begin{equation} \label{conv1}
	\psi_{m}(t) \to \psi(t) ~~\mbox{in}~~C[\sigma,T],~~\forall  T,\sigma ~~s.t.~~T>\sigma> 0,
\end{equation}
for some limit function $ \psi \in C(0,+\infty)$.

Since $s_0>0$, \eqref{conv1} yields
\begin{equation} \label{conv}
	r(tu_m)=r(t(s_mv_m)) \to \psi(t s_0)=:\hat{r}(t)~\mbox{as}~m\to \infty, ~~\mbox{in}~~C[\sigma,T],
\end{equation}
for any $ T>\sigma>0$. Furthermore, by weak lower semi-continuity of  $r(tu)$,  for every $t>0$, we have
\begin{equation}\label{wls}
	r(t u_0) \leq \liminf_{m\to \infty}r(t u_m).
\end{equation}
This and \eqref{conv} yield that for $t\geq 0$
\begin{equation}\label{hat2}
r(t u_0) \leq 	\hat{r}(t).
\end{equation}
\begin{prop}\label{prop5}
$u_0$ satisfies $\partial r(tu_0)/\partial t|_{t=1}>0$ for \eqref{N+} and  $\partial r(tu_0)/\partial t|_{t=1}<0$ for \eqref{N-}.	
\end{prop}
\begin{proof} 
We prove the assertion only for problem \eqref{N+}; problem \eqref{N-} can be handled in a similar way. 

Since  $r(tu_m)\leq \lambda$ for $t \in (0,1]$,  \eqref{conv} 
implies  $\hat{r}(t)\leq \lambda$ for $t \in (0,1]$.
Suppose, contrary to our claim,  that $\partial r(u_0 t)/\partial t|_{t=1}< 0$. Then due to property {\bf (a)} one has $t_{u_0}<1$, where $r(t_{u_0}u_0)=\max_{t>0} r(tu_0)$. Since $\hat{r}(t)\leq \lambda$ for $t \in (0,1]$ we deduce from \eqref{hat2} that $r(t_{u_0}u_0)\leq \lambda$. But by the assumption $\lambda<\lambda^*_{max}\leq \max_{t>0} r(tu_0)$, a contradiction.
\end{proof}


Now let us  conclude the proof of the theorem. 
By weak lower semi-continuity of  $\Phi_\lambda$ we have
\begin{equation}\label{lwc}
	-\infty<\Phi_\lambda(u_0)\leq \liminf_{m\to \infty}\Phi_\lambda(u_m)=\hat{\Phi}_\lambda^\pm.
\end{equation}
Hence, for  the existence of the minimizers $u^1_\lambda$ and $u^2_\lambda$ of \eqref{N-} and \eqref{N+}, respectively,  it is sufficient to show that $r(u_0)=\lambda$.

Suppose, contrary to our claim, that $r(u_0)<\lambda$. First consider  problem \eqref{N+}.
By Proposition \ref{prop5}, $\partial r(tu_0)/\partial t|_{t=1}>0$. Consequently property {\bf (a)} and the assumption 
$\lambda< \lambda^*_{max}\leq r(t_{u_0}u_0)=\max_{t>0} r(tu_0) $ yield that there is $t_1\in (1,t_{u_0})$ such that $r(t_1 u_0)=\lambda$ and $\partial r(tu_0)/\partial t|_{t=t_1}>0$. Moreover, since  $r(tu_0)<\lambda$ for $t \in [1,t_1)$, Corollary \ref{LP0} implies that $\partial \Phi_\lambda(tu)/\partial t<0$ for $t \in [1,t_1)$. Consequently
$$
\Phi_\lambda(t_1 u_0)<\Phi_\lambda(u_0).
$$
Thus by \eqref{lwc} we have $\Phi_\lambda(t_1 u_0)<\hat{\Phi}_\lambda$ and since $t_1 u_0 \in \mathcal{N}_\lambda^2$, we obtain a contradiction.

Consider now problem \eqref{N-}.  Since $\partial r(tu_0)/\partial t|_{t=1}<0$, $\lambda< \lambda^*_{max}$, assumption {\bf (a)} implies that  there is $t_2<1$ such that $r(t_2 u_0)=\lambda$ and $\partial r(tu_0)/\partial t|_{t=t_2}<0$. Note that by the weak lower semi-continuity of  $\Phi_\lambda$
\begin{equation}\label{Phim}
	\Phi_\lambda(t_2 u_0)\leq \liminf_{m\to \infty}\Phi_\lambda(t_2 u_m).
\end{equation}
Since \eqref{wls} and $r(t_2u_0)=\lambda$, $r(u_m)=\lambda$, assumption {\bf (a)}  implies that $r(tu_m)>\lambda$ for $t \in (t_2,1]$ and sufficiently large $m$. Then   Corollary \ref{LP0} implies that $\partial \Phi_\lambda(tu_m)/\partial t> 0$ for $t \in (t_2,1]$ and sufficiently large $m$. Consequently, $\Phi_\lambda(t_2 u_m)<\Phi_\lambda(u_m)$ and \eqref{Phim} yields
$$
\Phi_\lambda(t_2 u_0)\leq \liminf_{m\to \infty}\Phi_\lambda(t_2 u_m)\leq \liminf_{m\to \infty}\Phi_\lambda(u_m)=\hat{\Phi}_\lambda^1.
$$
Assume $\Phi_\lambda(t_2 u_0)=\hat{\Phi}^-_\lambda$, then taking into account that  $t_2u_0 \in \mathcal{N}_\lambda^1$  we see that $t_2 u_0$ is a minimizer of \eqref{N-} and we get the required. If $\Phi_\lambda(t_2 u_0)<\hat{\Phi}_\lambda^1$, then we obtain a contradiction and therefore  $r(u_0)=\lambda$. 
This completes the proof of the existence of the minimizers $u^1_\lambda$ and $u^2_\lambda$. 

Now in view of \eqref{partLam},  Theorem \ref{thm1} yields that $u^1_\lambda$ and $u^2_\lambda$ satisfy
\eqref{V1}. From Corollary \ref{LP} it follows that
$\partial^2 \Phi_\lambda(tu^1_\lambda)\partial t^2|_{t=1}<0$, $\partial^2 \Phi_\lambda(tu^2_\lambda)\partial t^2|_{t=1}>0$. To conclude the proof, it remains to show that $u^2_\lambda$ is the ground state of  \eqref{V1}. Assumption {\bf(a)} and  $\lambda \in (\lambda^\partial_{min}, \lambda^*_{max})$ yield that the equation $\partial \Phi_\lambda(su^1_\lambda)/\partial s=0$ has precisely two solutions $s_{min}<1$ and 
$s_{max}=1$ such that $\partial^2\Phi_\lambda(su^1_\lambda)/\partial s^2|_{s=s_{min}}>0$ and
$\partial^2\Phi_\lambda(su^2_\lambda)/\partial s^2|_{s=1}<0$. But then $\Phi_\lambda(u^1_\lambda)>\Phi_\lambda(s_{min}u^1_\lambda)\geq \hat{\Phi}_\lambda^2\equiv \Phi_\lambda(u^2_\lambda)$. Thus $u^2_\lambda$ is the minimizer of $\hat{\Phi}_\lambda:=  \min	\{\Phi_\lambda(u):u \in \mathcal{N}_\lambda\}$ and we get the required.

\end{proof}

We emphasize that  the   value  $\lambda^\partial_{min}$  has been used above only in order to allocate  the values $\lambda$ in $(\lambda^\partial_{min}, \lambda^*_{max})$ such that $\mathcal{N}_\lambda^1\neq \emptyset$, $ \mathcal{N}_\lambda^2 \neq \emptyset$.  In fact, the above proof of Theorem \ref{thm3} can be easily adapted to other assumptions on the behaviour of $r(tu)$ at  $t\to 0$ and $t\to \infty$. In particular, let us assume that for all $u\in W\setminus 0_n$ there holds
\begin{equation}\label{BB}
r(tu) \to 0~~\mbox{as}~~	t\to 0~~~\mbox{and}~~~ r(tu) \to -\infty~~ \mbox{as}~~	t\to \infty.
\end{equation}
It easy to see that then $\lambda^\partial_{min}=0$ and $\lambda^*_{min}=-\infty$. Moreover, now we have $\mathcal{N}_\lambda^1 \neq \emptyset$ for all $\lambda<\lambda^*_{max}$. It is easily seen that the above proof  of the existence  of the minimizer $u^1_\lambda$ of \eqref{N-} remains valid for all $\lambda<\lambda^*_{max}$, provided \eqref{BB} is satisfied. Thus we have 

\begin{cor}\label{cthm3}
Suppose the assumptions of Theorem \ref{thm3}  and \eqref{BB} hold. 
Then for every $\lambda<\lambda^*_{max}$ there exists a minimizer $u^1_\lambda$ of \eqref{N-} which satisfies  equation \eqref{V1}.  Furthermore, $\partial^2 \Phi_\lambda(tu^1_\lambda)\partial t^2|_{t=1}<0$ and $u^1_\lambda$ is the ground state of  \eqref{V1} if $\lambda\leq 0$.
\end{cor}
\begin{rem}
We stress that the obtained (in Theorem \ref{thm3}, Corollary \ref{cthm3}) solutions $u^1_\lambda, u^2_\lambda$ belong $W\setminus 0_n$. This means
that in the case $n>1$ it is possible that $u^1_{i,\lambda}=0$, $u^2_{j,\lambda}=0$ for some $i,j =1,2...n$ whereas $u^1_\lambda, u^2_\lambda \not\equiv 0_n$
(see also Remark \ref{RScV}).
\end{rem}

\par
\medskip
{\bf 6.1 Multiplicity nonnegative solutions for problems with a general convex-concave type nonlinearity}
\par
\medskip

In this subsection, using Theorem \ref{thm3} we obtain a  result on the existence of multiple sign-constant solutions for problems with a general convex-concave type nonlinearity and $p$-Laplacian.

Consider the following boundary value problem 
\begin{equation}
\label{pbGCC} \left\{
\begin{array}{l}
\
-\Delta_p u = \lambda |u|^{q-2}u+f(x,u),
~ x \in \Omega, \\ \\
\hspace*{0.2cm} u|_{\partial \Omega} = 0,
\end{array}
\right.
\end{equation}
where   $\Omega$ is a bounded domain in $\mathbb{R}^N$
 with smooth boundary $\partial \Omega$, $N\geq 1$,
$\Delta_p(\cdot)=\mbox{div}(|\nabla (\cdot)|^{p-2}\nabla (\cdot))$ is the p-Laplacian, $1<q<p<+\infty$, 
$f : \Omega \times\mathbb{R} \to \mathbb{R}$ is a Carath\'eodory function  such that $f(x,0)=0$, $f(x,\cdot) \in C(\mathbb{R},\mathbb{R})\cap C^1(\mathbb{R}\setminus 0,\mathbb{R})$ a.a. $x\in \Omega$, with primitive $F(x,u)=\int_0^uf(x,s)ds$.


We will suppose that $f$ satisfies the following conditions:
\par \noindent
\begin{quotation}
	$(1^o)$  $\exists \gamma_1, \gamma_2	\in (p,p^*)$, $\gamma_1\leq \gamma_2$ such that
	$$
	0<s\frac{\partial}{\partial s}f(x,s) \leq g_1(x)|s|^{\gamma_1-2} +g_2(x)|s|^{\gamma_2-2},~~s \in  \dot{\mathbb{R}},~\mbox{ a.a.}~  x\in \Omega,
	$$
 where $g_i \in L^{\beta_i}(\Omega)$, $g_i\geq 0$,  $\beta_i >p^*/(p^*-\gamma_i)$ if $N>p$ and $\beta_i>1$ if $N\leq p$, $i=1,2$.
		\end{quotation}
\begin{quotation}
	$(2^o)$ \, $\exists \theta>p$, $R_1>0$: $0<\theta F(x,s)\leq f(x,s)s$, a.a. $x\in \Omega$, if $|s|\geq R_1$;
\end{quotation}
\begin{quotation}
	$(3^o)$ \,for $\forall u \in \dot{\mathbb{R}}$ and a.a. $x\in \Omega$,  $\displaystyle {\frac{\partial }{\partial s}(s^{1-q} f(x,su))u}$ is a monotone increasing function on $(0,+\infty)$.
\end{quotation}

Example of $f$ satisfying $(1^o)$-$(3^o)$ is as follows: 
$f(x,s)=\sum_{i=1}^l f_i(x) |s|^{\gamma_i-2}s+\psi(s)$, $s \in \mathbb{R}$,  
where  $p<\gamma_1\leq...\leq\gamma_l<p^*$, $\psi \in C^1_0(\mathbb{R})$ such that $\psi(t)=0$ if $|t|<R_0$ for some $R_0>0$, $f_i\geq 0$ on $\Omega$, $f_i \in L^{\beta_i}(\Omega)$, $\beta_i$ is defined as in $(1^o)$, $i=1,...,l$.

By a solution of \eqref{pbGCC} we shall  mean a weak solution $u \in W:= W^{1,p}_0(\Omega)$.
Problem (\ref{pbGCC}) has a variational form with the Euler-Lagrange functional
\begin{equation}\label{pb11c}
\Phi_\lambda (u) = \frac{1}{p} \int |\nabla u|^{p} dx - \lambda\frac{1}{q} \int |u|^{q} dx -
 \int F(x,u)dx.
\end{equation}
Consider NG-Rayleigh's quotient
\begin{equation}\label{lamb2}
	r(u)=\frac{\int |\nabla u|^{p} dx-\int f(x,u)u dx}{\int |u|^{q} dx},~~u \in W\setminus 0
\end{equation}
and the corresponding fibered map
\begin{equation}\label{rapplCC}
	r(tu)=\frac{t^{p-q}\int|\nabla u|^{p} dx-t^{1-q}\int f(x,tu)u dx}{\int |u|^{q} dx},~~u \in W\setminus 0,~  t>0.
\end{equation}
Note that $(1^o)$ implies
\begin{equation}\label{10}
0<f(x,s) \leq g_1'(x)|s|^{\gamma_1-1} +g_2'(x)|s|^{\gamma_2-1}~~\mbox{ a.a.}~  x\in \Omega,
\end{equation}
where $g_i'=g_i/(\gamma_i-1)$, $i=1,2$. By Sobolev's embedding theorem and Holder's inequality for $u\in W$ and $i=1,2$ one has
\begin{equation}\label{estim}
	|\int_\Omega g_i'(x)|u|^{\gamma_i}dx |\leq C ||u||^{\gamma_i/p^*}(\int_\Omega |g'_i(x)|^{p^*/(p^*-\gamma_i)}
	dx)^\frac{(p^*-\gamma_i)}{p^*},
\end{equation}
where $C<+\infty$. Hence, $\Phi_\lambda$ and $r$ are well defined on $W$ and $W\setminus 0$, respectively.

Consider the extremal value
\begin{equation}\label{mulEX}
	\lambda^*_{max}=
	\inf_{u \in W\setminus 0}\sup_{t>0} \frac{t^{p-q}\int|\nabla u|^{p} dx-t^{1-q}\int f(x,tu)u dx}{\int |u|^{q} dx}.
\end{equation}

 We prove the following
\begin{thm}\label{thm4}
Assume $1<q<p<+\infty$ and  $(1^o)$-$(3^o)$ hold. Then 
$0<\lambda^*_{max}$ and for any $\lambda<\lambda^*_{max}$,  problem \eqref{pbGCC} admits a pair of non-trivial  weak solutions  $u_\lambda^{1,+}\geq 0 \geq u_\lambda^{1,-}$. Furthermore, when $\lambda \in (0,\lambda^*_{max})$ equation \eqref{pbGCC} has a second pair of non-trivial  weak solutions  $u_\lambda^{2,+}\geq 0 \geq u_\lambda^{2,-}$. Moreover,  
\begin{enumerate}
	\item $\partial^2 \Phi_\lambda(tu_\lambda^{1,\pm})\partial t^2|_{t=1}<0$,  $\partial^2 \Phi_\lambda(tu_\lambda^{2,\pm})\partial t^2|_{t=1}>0$, $\Phi_\lambda(u_\lambda^{2,\pm})<0$;
	\item if $\lambda\leq 0$, then one of the solutions $u^{1,+}_\lambda$ or $u^{1,-}_\lambda$ is a ground state of  \eqref{pbGCC};
	\item  if  $\lambda \in (0,\lambda^*_{max})$, then  one of the solutions $u^{2,+}_\lambda$ or $u^{2,-}_\lambda$ is a ground state of  \eqref{pbGCC}.
\end{enumerate}
\end{thm}
\par\noindent
{\it Remark.} Similar result on the existence of multiple sign-constant solutions  for problems with a general convex-concave type nonlinearity  has been obtained in \cite{AmAzPer, AmBrCer, LiWang}. 
However, our assumptions on function $f(x,s)$ are different from that were made in  \cite{AmAzPer, AmBrCer, LiWang}. In particular, in $(1^o)$ the functions $g_i$, $i=1,2$ permit be unbounded above, which causes difficulties in application of the super-sub solution method (cf. \cite{AmAzPer, AmBrCer, LiWang}). Furthermore, the presence of $p$-Laplacian with $p\neq 2$ in  \eqref{pbGCC} can
complicate the application of mountain pass theorem in order interval (cf \cite{LiWang}).

\begin{proof} We will obtain the proof applying Theorem \ref{thm3} and Corollary \ref{cthm3}.
Observe that  \eqref{lamb2} satisfies {\bf(a)} for any $u \in W\setminus 0$.
Indeed, compute
\begin{equation}\label{tmaxCC}
	\partial_t r(tu)=\frac{(p-q)t^{p-q-1}\int|\nabla u|^{p} dx-\int\partial_t( t^{1-q}f(x,tu))u dx}{\int |u|^{q} dx}.
\end{equation}
It can easily be checked that $(1^o)$ and \eqref{10} yield $\int\partial_t( t^{1-q}f(x,tu))u dx/t^{p-q-1} \to 0$ as $t\to 0$ for any $u\in W$. Hence by $(3^o)$  we get {\bf(a)}. Furthermore,  it follows that \eqref{rapplCC} satisfies condition \eqref{BB} of Corollary \ref{cthm3}. 

Let us show that {\bf (b)} holds. Assume that $(u_m) \subset  W\setminus 0$ such that $||u_m||=1$, $m=1,2,...$, and the set $(u_m)_{m=1}^\infty$ is weakly separated from $0 \in W$. Since  $(u_m)$ is bounded in $W$ and $W$ is the reflexive Banach space, by the Eberlein-Smulian theorem we may assume that $u_m \rightharpoondown u_0$ weakly in $W$  for some $u_0 \in W$. Furthermore, by Sobolev's embedding theorem $||u_m||_{L^d}<C_1<+\infty$ for $m=1,2,...$,   $1\leq d\leq p^*$, and $u_m \to u_0$ in $L^d(\Omega)$ for $d<p^*$. 
This implies that there exists $\delta_0>0$ such that $\int |u_m|^{q} dx>\delta_0$ for all $m=1,2,...$. 
Indeed,  if we assume the converse,   $u_{m_j} \to 0$ in $L^q(\Omega)$ for some subsequence  $m_j \to +\infty$, then $u_{m_j} \rightharpoondown 0$ weakly in $W$ which contradicts the assumption that $(u_m)_{m=1}^\infty$ is weakly separated from $0 \in W$. Hence by $(1^o)$, \eqref{10} and \eqref{estim} we have: for any t>0 
$$
|r(tu_m)|\leq \delta^{-1}_0 (t^{p-q} +t^{1-q}\int |f(x,tu_m)u_m |dx)\leq C_2 t^{p-q}+C_3t^{\gamma_2-q}, 
$$
and 
\begin{align*}
	|\partial_t r(tu_m)|&\leq \delta^{-1}_0 ((p-q)t^{p-q-1} +(q-1)t^{-q}\int |f(x,tu_m)u_m |dx+\\
	& t^{1-q}\int |\partial_t f(x,tu_m)u_m |dx)\leq C_4 t^{p-q-1}+C_5t^{\gamma_2-q-1},
\end{align*}
where  $C_2,...,C_5$ do not depend on $t>0$ and $m=1,2,...$. Thus we get {\bf (b)}. 

Let us verify {\bf (c)}. Denote $t_m:=t_{v_m, max}$ where $||v_m||=1$, $m=1,2,...$. In view of \eqref{tmaxCC} we have 
\begin{equation}
	(p-q)t_m^{p-q-1}-\int\partial_t( t_m^{1-q}f(x,t_mv_m))v_m dx=0.
\end{equation}
 Hence using $(1^o)$, \eqref{10} and \eqref{estim} we  get 
$$
(p-q)t_m^{p-q-1}-c_1 t_m^{\gamma_2-q-1}-c_2 t_m^{\gamma_1-q-1}\leq 0,
$$
 which is impossible if $t_m \to 0$ since by the assumption $p<\gamma_2\leq \gamma_1$. Thus {\bf (c)} holds.

Let us verify {\bf (d)}. Observe that \eqref{10} and \eqref{estim} imply
\begin{equation}\label{EstCC}
r(tu)\geq \frac{t^{p-q}\int|\nabla u|^{p} dx-C_1't^{\gamma_2-q}||u||^{\gamma_2} -C_2't^{\gamma_1-q}||u||^{\gamma_1} }{||u||^{q}},
\end{equation}
for $t>0$, $u\in W\setminus 0$. 	
Assume $(s_m v_m) \subset \mathcal{N}_\lambda^{sc}$, $\lambda \in \mathbb{R}$ such that $||v_m||_W=1$, $v_m \rightharpoondown 0$ weakly in $W$ and $\inf_m s_m >\delta>0$. Then, as it has been shown above $v_m \to 0$ in $L^d(\Omega)$ for $d\in (1, p^*)$. Hence, since $\gamma_2\leq \gamma_1$, by \eqref{EstCC} we have
\begin{align*}
	s_m^{\gamma_1-q}\geq\frac{\delta_0^{p-q}-\lambda ||v_m||_{L^{q}}^{q}}{C'_1\delta_0^{\gamma_2-\gamma_1} ||v_m||_{L^{\gamma_2}}^{\gamma_2}+C_2'||v_m||_{L^{\gamma_1}}^{\gamma_1}} \to \infty. 
\end{align*}
Thus {\bf (d)} also holds.

Let us now verify conditions (1), (2) of Theorem \ref{thm3}. For $u \in \mathcal{N}_\lambda $ we have 
$$
\Phi_\lambda(u)=\frac{(\theta-p)}{p} \int |\nabla u|^{p} dx - \lambda\frac{(\theta-q)}{q} \int |u|^{q} dx -
\int(\theta F(x,u)-f(x,u)u)dx.  
$$ 
Hence $(2^o)$ and Sobolev's embedding theorem yield
$$
\Phi_\lambda(u)\geq  \frac{(\theta-p)}{p} ||u||^p-\frac{\lambda(\theta-q)}{q} ||u||^q,
$$
for $||u||>R_1$. Since $q<p$, this implies $\Phi_\lambda(u) \to +\infty$ as $||u|| \to +\infty$ and thus  condition 
(1) of Theorem \ref{thm3} holds. 

By the above,  the functionals $\int F(x,u)dx$, $\int f(x,u)udx$,   $\int|u|^q dx$ are weakly continuous on $W$. Hence, since 
$ \int |\nabla u|^{p} dx$ is a weakly lower semi-continuous functional on $W$ we derive that  \eqref{pb11c} and \eqref{lamb2} satisfy (2) of Theorem \ref{thm3}.

Note that \eqref{EstCC}, $(1^o)$ and  Sobolev's embedding theorem yield
\begin{align}
	\lambda^*_{max}\geq \inf_{u \in W: ||v||=1}\sup_{t>0}&\frac{t^{p-q}||v||^{p}-c_1't^{\gamma_2-q}||v||^{\gamma_2} -C_1't^{\gamma_1-q}||v||^{\gamma_1} }{||v||^q}=\\
	\max_{t>0}&\{t^{p-q}-c_1't^{\gamma_2-q} -C_1't^{\gamma_1-q}\}>0. 
\end{align}
for some $c_1',C_1'>0$. Hence  $\lambda^*_{max}>0$.
Thus, we see that the functional $\Phi_\lambda(u)$ satisfies all assumptions of Theorem \ref{thm3} and Corollary \ref{cthm3}. 

In order to obtain solutions  $u_\lambda^{1,+}\geq 0 \geq u_\lambda^{1,-}$ and $u_\lambda^{2,+}\geq 0 \geq u_\lambda^{2,-}$, we truncate and reflect $f(x,u)$ as follows
\begin{align}
\label{trunc}
	f^\pm(x,u)= 
	\begin{cases}
		f(x,u)~~~\mbox{if}~~~\pm u\geq 0,\\
		f(x,-u)~~~\mbox{if}~~~\pm u< 0.
	\end{cases}
\end{align}
Let $F^\pm(x,u)$ denote the primitive of $f^\pm(x,u)$  and consider the Euler-Lagrange functional
\begin{equation}\label{pb11cC}
\Phi_\lambda^\pm (u) = \frac{1}{p} \int |\nabla u|^{p} dx - \lambda\frac{1}{q} \int |u|^{q} dx -
 \int F^\pm(x,u)dx.
\end{equation}
Then as above $\Phi_\lambda^\pm (u) \in  C^1(\dot{W}, \mathbb{R})$, $\nabla_t \tilde{\Phi}_\lambda^\pm \in C^1((\dot{\mathbb{R}}^+)^n \times \dot{W}, \mathbb{R}^n)$ it satisfies all the assumptions of Theorem \ref{thm3} and Corollary \ref{cthm3}. Thus there exist weak solutions $u_\lambda^{1,\pm}, u_\lambda^{2,\pm}\in W^{1,p}_0(\Omega)$ of 
$$
-\Delta_p u = \lambda |u|^{q-2}u+f^\pm(x,u),
$$
for $\lambda<\lambda^*_{max}$ and $\lambda \in (0,\lambda^*_{max})$, respectively. 
Since $\Phi_\lambda^\pm(|u|)= \Phi_\lambda^\pm(u)$ we may assume that the minimizers $u_\lambda^{1,+}, u_\lambda^{2,+}$ of $\hat{\Phi}_\lambda^{j,+}:=  \min	\{\Phi_\lambda^+(u):~u \in \mathcal{N}_\lambda^j\}$, $j=1,2$, respectively, are non-negative, whereas the minimizers $u_\lambda^{1,-}, u_\lambda^{2,-}$ of $\hat{\Phi}_\lambda^{j,-}:=  \min	\{\Phi_\lambda^-(u):~u \in \mathcal{N}_\lambda^j\}$, $j=1,2$, respectively, are non-positive.  Now taking into account \eqref{trunc}  we get that the functions $u_\lambda^{1,\pm}, u_\lambda^{2,\pm}$ in fact are weak solutions of the original problem \eqref{pbGCC}. Finally, the assertions (1)-(3) of the theorem follow from Theorem \ref{thm3} and Corollary \ref{cthm3}.   
\end{proof}

\subsection*{Acknowledgement} The authors would like to acknowledge the considerable benefit obtained from discussions with V. Bobkov and  D. Motreanu.

\end{document}